\documentclass[12pt]{amsart}
\usepackage{verbatim,amssymb,latexsym,amscd}
\usepackage[all]{xy}
\usepackage{color}

\usepackage[colorlinks,linkcolor=blue,citecolor=blue,urlcolor=red]{hyperref}

\newcommand{\ie}{{\it i.e. }}
\newcommand{\cf}{{\it cf. }}
\newcommand{\eg}{{\it e.g. }}
\newcommand{\loccit}{{\it loc. cit.}}
\newcommand{\resp}{{\it resp. }}

\newcommand{\A}{\mathbf{A}}

\newcommand{\C}{\mathbf{C}}

\renewcommand{\P}{\mathbf{P}}
\newcommand{\Q}{\mathbf{Q}}

\newcommand{\Z}{\mathbf{Z}}
\newcommand{\sA}{\mathcal{A}}
\newcommand{\sB}{\mathcal{B}}
\newcommand{\sC}{\mathcal{C}}
\newcommand{\sD}{\mathcal{D}}
\newcommand{\sE}{\mathcal{E}}
\newcommand{\sF}{\mathcal{F}}
\newcommand{\sG}{\mathcal{G}}
\newcommand{\sH}{\mathcal{H}}
\newcommand{\sI}{\mathcal{I}}

\newcommand{\sK}{\mathcal{K}}
\newcommand{\sL}{\mathcal{L}}
\newcommand{\sM}{\mathcal{M}}
\newcommand{\sO}{\mathcal{O}}
\newcommand{\sR}{\mathcal{R}}
\newcommand{\sS}{\mathcal{S}}
\newcommand{\sT}{\mathcal{T}}
\newcommand{\sU}{\mathcal{U}}
\newcommand{\sX}{\mathcal{X}}

\newcommand{\bG}{\mathbb{G}}
\newcommand{\bH}{\mathbb{H}}
\newcommand{\bL}{\mathbb{L}}

\newcommand{\uC}{\underline{C}}
\newcommand{\Spec}{\operatorname{Spec}}
\newcommand{\Ker}{\operatorname{Ker}}
\newcommand{\Coker}{\operatorname{Coker}}
\newcommand{\IM}{\operatorname{Im}}

\newcommand{\DM}{\operatorname{\bf DM}}

\newcommand{\Hom}{\operatorname{Hom}}
\newcommand{\End}{\operatorname{End}}
\newcommand{\CM}{\operatorname{\bf CM}}
\newcommand{\Mod}{\text{\rm Mod--}}

\newcommand{\Pic}{\operatorname{Pic}}

\newcommand{\Nis}{{\operatorname{Nis}}}

\newcommand{\uHom}{\operatorname{\underline{Hom}}}
\newcommand{\Ext}{\operatorname{Ext}}
\newcommand{\Tor}{\operatorname{Tor}}

\newcommand{\Ab}{\operatorname{\bf Ab}}

\newcommand{\car}{\operatorname{char}}
\newcommand{\rat}{{\operatorname{rat}}}

\newcommand{\ind}{{\operatorname{ind}}}

\newcommand{\nr}{{\operatorname{nr}}}

\newcommand{\Sm}{\operatorname{\bf Sm}}

\newcommand{\proj}{{\operatorname{proj}}}

\newcommand{\SmCor}{\operatorname{\bf SmCor}}
\newcommand{\Cor}{\operatorname{\bf Cor}}
\newcommand{\BFC}{\operatorname{\bf BFC}}
\newcommand{\PST}{\operatorname{\bf PST}}
\newcommand{\NST}{\operatorname{\bf NST}}
\newcommand{\HI}{\operatorname{\bf HI}}
\newcommand{\PHI}{\operatorname{\bf PHI}}
\newcommand{\Tot}{\operatorname{Tot}}

\newcommand{\Chow}{\operatorname{\bf Chow}}
\newcommand{\gm}{{\text{\rm gm}}}
\newcommand{\equi}{{\text{\rm equi}}}

\renewcommand{\o}{{\text{\rm o}}}
\renewcommand{\b}{{\text{\rm b}}}
\newcommand{\op}{{\text{\rm op}}}
\newcommand{\eff}{{\text{\rm eff}}}

\newcommand{\by}[1]{\overset{#1}{\longrightarrow}}
\newcommand{\yb}[1]{\overset{#1}{\longleftarrow}}
\newcommand{\iso}{\by{\sim}}
\newcommand{\osi}{\yb{\sim}}
\newcommand{\inj}{\hookrightarrow}
\newcommand{\Inj}{\lhook\joinrel\longrightarrow}
\newcommand{\surj}{\rightarrow\!\!\!\!\!\rightarrow}
\newcommand{\Surj}{\relbar\joinrel\surj} 
\newcommand{\colim}{\varinjlim}
\renewcommand{\lim}{\varprojlim}
\newcommand{\hocolim}{\operatornamewithlimits{hocolim}}

\newcommand{\codim}{\operatorname{codim}}

\renewcommand{\qed}{\hfill $\Box$\medskip}

\renewcommand{\phi}{\varphi}
\renewcommand{\epsilon}{\varepsilon}

\newcounter{spec}
\newenvironment{thlist}{\begin{list}{\rm{(\roman{spec})}}%
{\usecounter{spec}\labelwidth=20pt\itemindent=0pt\labelsep=10pt}}%
{\end{list}}%

\newtheorem{Th}{Theorem}

\swapnumbers

\newtheorem{thm}{Theorem}[subsection]
\newtheorem{lemma}[thm]{Lemma}
\newtheorem{prop}[thm]{Proposition}
\newtheorem{cor}[thm]{Corollary}

\theoremstyle{definition}

\newtheorem{defn}[thm]{Definition}

\newtheorem{rk}[thm]{Remark}
\newtheorem{rks}[thm]{Remarks}

\newtheorem{ex}[thm]{Example}

\numberwithin{equation}{section}

\entrymodifiers={!!<0pt,0.7ex>+}

\begin{document}
\title{Birational motives, II: Triangulated birational motives}
\author{Bruno Kahn}
\address{IMJ-PRG\\ 4 Place Jussieu\\75010 Paris\\France}
\email{bruno.kahn@imj-prg.fr}
\author{R. Sujatha}
\address{University of British Columbia\\Vancouver, BC
V6T1Z2\\Canada}
\email{sujatha@math.ubc.ca}
\date{April 18, 2016}
\thanks{The first author acknowledges the support of Agence Nationale de la Recherche (ANR) under reference ANR-12-BL01-0005 and the second author  that of NSERC Grant 402071/2011. Both authors acknowledge the support of CEFIPRA project 2501-1.}
\begin{abstract}
We develop birational versions of Voevodsky's triangulated categories of motives over a field, and relate them with the pure birational motives studied in \cite{birat-pure}. We also get an interpretation of unramified cohomology in this framework, leading to ``higher derived functors of unramified cohomology''.
\end{abstract}
\maketitle

\tableofcontents

\section*{Introduction}

This is the last part of our project on birational motives: it develops a birational analogue to Voevodsky's theory of triangulated motives. We work over a field $F$. A summary of our results may be read in the following commutative diagram of categories:

\begin{equation*}\label{eq.intr}
\xymatrix{
\Chow^\eff\ar[d]\ar[r]&\DM_\gm^\eff\ar[d]\ar[r] &\DM^\eff\ar[d]^{\nu_{\le 0}}\\
\Chow^\o\ar[r]&\DM_\gm^\o\ar[r]& \DM^\o.
}
\end{equation*}

In the top row, $\Chow^\eff$ is the category of effective Chow motives over $F$ (with integer coefficients), $\DM_\gm^\eff$ is Voevodsky's triangulated category of effective geometric motives \cite[\S 2]{voetri} and $\DM^\eff$ is an unbounded version of his triangulated category of motivic complexes \cite[\S 3]{voetri}.

In the bottom row, $\Chow^\o$ is the category of pure birational motives introduced in \cite[Def. 2.3.5]{birat-pure}: if $X,Y$ are smooth projective varieties with motives $h^\o(X),h^\o(Y)\in \Chow^\o$, we have an isomorphism $\Hom(h^\o(X),h^\o(Y))=CH_0(Y_{F(X)})$. The categories $\DM_\gm^\o$ and $\DM^\o$ are respectively obtained from $\DM_\gm^\eff$ and $\DM^\eff$ by inverting birational morphisms (Definition \ref{d4.1}).

When $F$ is perfect, the functors in the top row are full embeddings as a consequence of Voevodsky's main theorems on homotopy invariant pretheories \cite{voepre}; by the same theorems, $\DM^\eff$ enjoys a canonical ``homotopy $t$-structure''. All these facts turn out to be true also in the bottom row, without assuming $F$ perfect (Theorem \ref{t3.0}); their proofs are much more elementary and don't rely on the results of \cite{voepre}.

The heart of the homotopy $t$-structure on $\DM^\o$ is the category $\HI^\o$ of \emph{birational presheaves with transfers}: these are simply the presheaves with transfers of \cite[Def. 3.1.1]{voetri} which invert birational morphisms. This abelian category has truly excellent properties:

\begin{itemize}
\item $\HI^\o$ is a category of modules over an additive category; as such it has enough injectives, enough projectives, exact infinite direct sums and (quite unusually) exact infinite direct products (\cf Proposition \ref{p5.3}).
\item A birational presheaf with transfers is automatically a Nisnevich sheaf, and is homotopy invariant; it has no higher Nisnevich cohomology (Proposition \ref{l14.1}).
\end{itemize}

The functor $\nu_{\le 0}$ has a right adjoint $i^\o$, which in turn has a right adjoint $R_\nr$. When $F$ is perfect, one can compare the $t$-structures of $\DM^\eff$ and $\DM^\o$: then $\nu_{\le 0}$ is right $t$-exact, $i^\o$ is $t$-exact and $R_\nr$ is left $t$-exact. In particular, if $\sF$ is a homotopy invariant Nisnevich sheaf with transfers, the complex $R_\nr \sF$ is concentrated in degrees $\ge 0$. We compute its $0$-cohomology sheaf $R^0_\nr \sF$ as follows:

\begin{Th}[\cf Theorem \protect{\ref{t6.2}}] For any smooth connected $F$-variety $X$, one has
\[R^0_\nr\sF(X) =  \Ker\left(\sF(K)\by{(\partial_v)} \prod_v \sF_{-1}(F(v))\right)\]
where $K=F(X)$ and $v$ runs through all the $F$-divisorial valuations on $K$. Here $\sF_{-1}$ denotes Voevodsky's contraction of $\sF$, \cf \S \ref{s.contr}. 
\end{Th}

Thus we recover unramified cohomology in the sense of Colliot-Th\'el\`ene--Ojanguren \cite{ct}: this was one of the initial aims of our project, which was not achieved in the 2002 preprint version \cite{birat}. This also shows that unramified cohomology has, in some fashion, \emph{higher derived functors} which define new birational invariants: these functors are partly studied in \cite{rqnr}.

Given the long period of gestation of this work, there have been other expositions of triangulated birational motives, notably in \cite{motiftate} and \cite{kl}: they are essentially independent from the present one. We would like to finish this introduction by pointing a mistake in the initial version:

In \cite{birat}, Theorem 7.7, Corollary 7.8 and Corollary 7.9 c) are false: see Remark \ref{r5.2} below as concerns Theorem 7.7. The ``upper half'' of Corollary 7.9 c) remains true, as in Theorem \ref{t3.0} e). The contents of \S \ref{r5.1} may be viewed as a comment on this mistake.

\subsection*{Acknowledgements} While writing this paper, we benefited from discussions and exchanges with a large number of colleagues. We would like to especially thank Joseph Ayoub, Alexander Beilinson,  Fr\'ed\'eric D\'eglise, Eric Friedlander, Dennis Gaitsgory, Jens
Hornbostel, Annette Huber-Klawitter, Bernhard Keller, Marc Levine,
Geor\-ges Maltsiniotis, Fabien Morel, Amnon Neeman, Jo\" el Riou, Rapha\"el Rou\-quier, Vla\-di\-mir Voevodsky and Chuck Weibel. We also thank the referees for helpful comments.

\subsection{Notation} $F$ is the base field. All varieties are
$F$-varieties and all morphisms are $F$-morphisms. If $X$ is an irreducible variety,
$\eta_X$ denotes its generic point. We write $\Sm$ for the category of smooth varieties (= smooth separated $F$-schemes of finite type).

If $A$ is an abelian group and $p$ is a prime number, we write $A[1/p]:=A\otimes \Z[1/p]$. If $\sC$ is a category and $X,Y$ are objects of $\sC$, we write $\sC(X,Y)$ or $\Hom_\sC(X,Y)$ for the set of morphisms from $X$ to $Y$, depending on which of these notations is most convenient. If $\sA$ is an additive category, we write $\sA^\natural$ for its pseudo-abelian envelope;  if $p$ is a prime number, we write $\sA[1/p]$ for the category with the same objects and Hom groups given by $\sA[1/p](A,B)=\sA(A,B)[1/p]$.

\section{Birational sheaves with transfers}\label{s1.5}

In this section, we study modules over the category $\BFC$ of \emph{birational finite correspondences}, which is (Definition \ref{d1.1}) the localisation of Voevodsky's category of finite correspondences obtained by inverting birational morphisms. This category has several incarnations; elementary ones are given in Theorem \ref{p1.1}, while less elementary ones will be given in Proposition \ref{p2.6} and Theorem \ref{t2.2}. This is what gives a pivotal r\^ole to the category $\HI^\o=\Mod\BFC$ of \emph{birational presheaves with transfers}. These presheaves turn out to be automatically sheaves for the Nisnevich topology, and homotopy invariant; moreover they are acyclic for Nisnevich cohomology (Proposition \ref{l14.1}). When $F$ is perfect, they can be characterised as those homotopy invariant Nisnevich sheaves with transfers whose Voevodsky  contraction vanishes (Proposition  \ref{p2.3}).

\subsection{Birational finite correspondences} \label{s1.0} We start from the category  $\SmCor$ introduced by Voevodsky in \cite{voetri}; its objects are smooth $F$-varieties and its morphisms are \emph{finite correspondences}: for $X,Y\in \SmCor$, $\SmCor(X,Y)$ is the free abelian group $c(X,Y)$ with basis the set of closed integral subschemes of $X\times Y$ which are finite and surjective over a connected component of $X$. In \cite{mvw}, the notation was changed from $\SmCor$ to $\Cor$: we retain the original notation to avoid confusion with Chow correspondences between smooth projective varieties, which are also used here (see \S \ref{s1.1}). In contrast to the latter, finite correspondences compose ``on the nose'' \cite[Lemma 1.7]{mvw}; together with the product of varieties and cycles, they make $\SmCor$ an additive $\otimes$-category. The ``graph'' functor $\Sm\to \SmCor$ is the identity on objects and sends a morphism to its graph.

\begin{defn}\label{d1.1} The category of \emph{birational finite correspondences} is $\BFC=S_b^{-1}\SmCor$, where $S_b$ is the class of (graphs of) birational morphisms.\footnote{\label{f1}More correctly, a morphism $f:X\to Y$ is in $S_b$ if it is dominant and its restriction to any connected component of $X$ yields a birational morphism to some connected component of $Y$; we would get the same localisation by using the class $S_\o$ of dense open immersions. See \cite[1.7]{Birat} and \cite[2.1]{birat-pure} for detailed discussions.}
\end{defn}

We note that $S_b$ is closed under disjoint unions and products; hence, by \cite[Prop. A.1.2 and Th. A.3.1 ]{birat-pure}, the additive and tensor structures of $\SmCor$ pass to $\BFC$.

We shall use other incarnations of $\BFC$. It is helpful to use the category of \cite[Def. 2.25]{mvw}: its objects are smooth varieties and morphisms between two smooth varieties $X,Y$ are given by
\begin{equation}\label{eq1.3}\begin{CD}
h_0(X,Y) = \Coker(c(X\times \A^1,Y)@>{i_1(X)^*-i_0(X)^*}>> c(X,Y))
\end{CD}\end{equation}
where $i_t:\Spec k\to \A^1$ is the inclusion of the point $t$ and $i_t(X)=i_t\times 1_X$. We denote this category by  $\sH(\SmCor)$. 

\begin{lemma}\label{l1.1} Let $S_h$ be the class of projections $\pi_X:X\times \A^1\to X$. Then $\sH(\SmCor)$ is isomorphic to $S_h^{-1} \SmCor$.
\end{lemma}

\begin{proof} Let us show that $\sH(\SmCor)$ and $S_h^{-1} \SmCor$ have each other's universal property. Let $F:\SmCor\to \sC$ be a functor. If $F$ factors through $S_h^{-1} \SmCor$, then $F(i_0(X))=F(i_1(X))$ for any $X$, as both are inverse to $F(\pi_X)$. Hence $F$ factors through $\sH(\SmCor)$. On the other hand, $\pi_X$ is invertible in $\sH(\SmCor)$ (see comment after \cite[Def. 2.25]{mvw}). Hence, if $F$ factors through $\sH(\SmCor)$, it also factors through $S_h^{-1} \SmCor$.
\end{proof}

\begin{thm}\label{p1.1} In the diagram
\begin{equation}\label{eq1.1}
\begin{CD}
\BFC= S_b^{-1} \SmCor@>>> S_b^{-1}\sH(\SmCor)\\
@VVV @VVV\\
S_r^{-1} \SmCor@>>> S_r^{-1}\sH(\SmCor)
\end{CD}
\end{equation}
all functors are isomorphisms of categories. Here $S_r$ denotes the class of [graphs of] stable birational morphisms.\footnote{A morphism between connected smooth varieties is \emph{stably birational} if it is dominant and induces a purely transcendental extension of function fields; this is extended to nonconnected smooth varieties as in footnote \ref{f1}.}
\end{thm}

\begin{proof}   It follows from Lemma \ref{l1.1}  and \cite[Lemma 1.7.1]{Birat} that the bottom horizontal functor is an isomorphism of categories; on the other hand, the vertical functors are isomorphisms of categories by \cite[Th. 1.7.2]{Birat}. 
\end{proof}

We shall get further descriptions of $\BFC$ in Theorem \ref{t2.2} and Remark \ref{r1.2}.

\subsection{Review of modules over additive categories}\label{s.not} 

We refer to \cite[\S 1]{ak} and \cite[App. A]{somekawa}  for this additive version of \cite[I.5.3]{SGA4}. Let $\sA$ be an (essentially small)  additive category: we recall the fully faithful additive Yoneda functor 
\[\sA\by{y}\Mod\sA\]
where $\Mod\sA$ is the category of right $\sA$-modules (= contravariant additive functors from $\sA$ to abelian groups).  Let us call an object of $\Mod\sA$
\emph{representable} if it is in the image of the Yoneda functor
$y:\sA\to \Mod\sA$ and
\emph{free} if it is a direct sum of representable objects. Also recall that an object $X$ of a category $\sC$ is \emph{compact} if $\sC(X,-)$ commutes with arbitrary direct limits (representable in
$\sC$). We then have the following general facts (\cite[Prop. 1.3.6 and A.1.4]{ak}, see also \cite[Exp. 1, p. 97, Ex. 8.7.8]{SGA4} in the nonadditive case):

\begin{prop}\label{p5.3}   
a) The category $\Mod\sA$ is abelian, has enough projectives and enough
injectives and admits arbitrary direct and inverse limits. Filtering
direct limits and products are exact.\\ 
b) Let $\sA^\natural$ be the pseudo-abelian envelope of $\sA$; then $y$ extends to a full embedding $\sA^\natural\inj\Mod\sA$. Its image consists of all compact objects, and these objects are projective. \\
c) Any free object is projective and there are enough free objects.
\end{prop}

If $f:\sA\to \sB$ is an additive functor, it induces a triple of adjoint functors, with $(f^*\sF)(A)=\sF(f(A))$ for $\sF\in \Mod\sB$): 
\[\Mod\sA\begin{smallmatrix}f_!\\\longrightarrow\\f^*\\\longleftarrow\\f_*\\\longrightarrow \end{smallmatrix} \Mod \sB\]
(each functor is left adjoint to the one below it), and $f_!$ naturally commutes with $f$ relatively to the Yoneda embeddings. If $\sF\in \Mod\sA$, the unit (\resp counit) morphism
\begin{equation}\label{equnit}
\sF\by{\eta} f^*f_!\sF\quad (\text{\resp } f^*f_*\sF\by{\epsilon} \sF)
\end{equation}
is then universal among morphisms from $\sF$ to (\resp to $\sF$ from) presheaves of the form $f^*\sG$.
Note that $f^*$, hence $f_!$, is an equivalence of categories for $f:\sA\to \sA^\natural$ the canonical embedding.

The functor $f^*$ is fully faithful in two cases: when $f$ is a localisation, or when $f$ is full and essentially surjective. These facts are left to the readers as exercises. This persists if one further passes to pseudo-abelian envelopes.  

If $\sA$ is monoidal, its tensor structure extends to $\Mod\sA$ in such a way that the Yoneda embedding is monoidal; if $f:\sA\to \sB$ is a monoidal functor between (additive) monoidal categories, then $f_!$ is monoidal \cite[A.8, A.12]{somekawa}.

\subsection{Elementary properties of birational sheaves}\label{s2.1} 

With the notation of \S \ref{s.not}, the category $\PST$ of presheaves with transfers of \cite[Def. 3.1.1]{voetri} or \cite[Def. 2.1]{mvw} is none other than $\Mod\SmCor$. 

\begin{defn}\label{hi0} We denote by $\HI^\o$ the full subcategory of $\PST$
consisting of those presheaves $\sF$ that are birationally invariant, \ie
such that
$\sF(X)\iso
\sF(U)$ for any dense open immersion $j:U\to X$. We call an object of $\HI^\o$ a birationally invariant homotopy invariant presheaf with transfers, or for short, a \emph{birational sheaf} [with transfers]. 
\end{defn}

By definition, the obvious functor
\[\HI^\o\to
\Mod \BFC=\Mod \BFC^\natural\] 
is an isomorphism of categories; Proposition \ref{p5.3} therefore applies. 

Let $\PHI$ denote the full subcategory of $\PST$ consisting of homotopy invariant presheaves with transfers. We may identify $\PHI$ with $\Mod\sH(\SmCor)$ (see \S \ref{s1.0}).  In view of Proposition \ref{p5.3} a), the string of functors $\SmCor\by{\alpha} \sH(\SmCor)\by{\beta} \BFC$ yields a naturally commutative diagram of categories (ibid.)
\begin{equation}\label{eq1.0}
\begin{CD}
\SmCor @>L>> \PST\\
@V\alpha VV @V\alpha_!VV\\
\sH(\SmCor)@>h_0>> \PHI\\
@V\beta VV @V\beta_!VV\\
\BFC @>h_0^\o>> \HI^\o.
\end{CD}
\end{equation}

Here we follow the notation of \cite[p. 199 and 207]{voetri} for the two top Yoneda functors (in \cite{mvw}, the notation $L$ is replaced by $\Z_{tr}$): for $Y\in \Sm$, $h_0(Y)$ is the presheaf $X\mapsto h_0(X,Y)$ (\cf \eqref{eq1.3}). As in \cite{voetri}, we also write $\alpha_!\sF =h_0(\sF)$ for $\sF\in \PST$. We will sometimes  write $\sF^\o=\beta_!\alpha_!\sF$, so that the canonical map $\sF\to \alpha^*\beta^*\sF^\o$ is universal among morphisms from $\sF$ to birational presheaves. If $\sF\in \PHI$, we have 
\[\sF^\o=\beta_!\alpha_!\alpha^*\sF=\beta_!\sF.\]

We shall also use the category of Nisnevich sheaves with transfers \cite[Def. 3.1.1]{voetri}, \cite[Lect. 13]{mvw}, that we denote here by $\NST$:  by definition, this is the full subcategory of $\PST$ formed by those presheaves with transfers which are sheaves in the Nisnevich topology. We also write $\HI=\NST\cap \PHI$ for the category of homotopy
invariant Nisnevich sheaves with transfers (see \cite[Prop.
3.1.13]{voetri}).  Recall the exact sheafification functor   \cite[Th. 3.1.4]{voetri}
\begin{equation}\label{eq3.2}
a:\PST\to \NST
\end{equation}
which is left adjoint of the inclusion functor $k:\NST\inj \PST$.

The following innocent-looking lemma turns out to be very powerful, and justifies the terminology ``birational sheaf''.

\begin{lemma}\label{l14.0} a) Any presheaf of sets $\sF$ on $\Sm$
  which transforms coproducts into products and is birationally
  invariant in the sense that $\sF(X)\iso \sF(U)$ for any dense open
  immersion $U\inj X$ is a sheaf for the Nisnevich topology.\\ 
b) If $\sF$ is moreover a (pre)sheaf of abelian groups, then
$H^i_\Nis(X,\sF)=0$ for all $X\in \Sm$ and all $i\ne 0$. 
\end{lemma}

\begin{proof} a) This follows from  \cite[p. 96, Prop. 1.4]{mv}. b) This
follows from \cite[Lemma 1.40]{riou}.\end{proof}

\begin{prop} \label{l14.1}  a) One has $\HI^\o\subset \HI$.\\
b) For any $\sF\in \HI^\o$ and any $X\in\Sm$, one has
$H^i_\Nis(X,\sF)=0$ for all $i\ne 0$.
\end{prop}

\begin{proof} a) Let $\sF\in \HI^\o$. By Lemma \ref{l14.0} a), $\sF$ is a sheaf in the Nisnevich
topology. The fact that it is homotopy invariant follows from \cite[Th. 1.7.2]{Birat}. b) merely repeats Lemma \ref{l14.0} b).\footnote{In particular it is not necessary
  to invoke \protect\cite[Th. 3.1.12]{voetri} as we did in \cite{birat}:  we owe this remark to
  Jo\"el Riou.} 
\end{proof}

\subsection{Contractions}\label{s.contr} Recall the following definition from \cite[p. 96]{voepre} or
\cite[Lect. 23]{mvw}:

\begin{defn} Let $\sF\in \PST$. Then $\sF_{-1}\in \PST$ is the presheaf with transfers defined by
\[\sF_{-1}(X)=\Coker\left( \sF(X\times \A^1)\to  \sF(X\times (\A^1-\{0\}))\right).\] 
This is the \emph{contraction} of $\sF$.
\end{defn}

Note that if $\sF\in \PHI$, $\sF_{-1}(X)$ is a functorial direct summand of $\sF(X\times(\A^1-\{0\}))$ because the map $\sF(X)\to \sF(X\times(\A^1-\{0\}))$ has a section given by $1\in \A^1-\{0\}$; in particular, $\sF_{-1}\in \PHI$. If $\sF\in \HI$, this argument shows that $\sF_{-1}\in \HI$,  hence $\sF\mapsto \sF_{-1}$ defines an exact endofunctor of $\HI$.

We may extend $\sF$ to smooth separated 
$F$-schemes essentially of finite type by taking direct limits over open sets, in a standard
way; in particular, we write $\sF(K)$ for $\sF(\Spec K)=\colim \sF(U)$ if $K$ is the function
field of a smooth irreducible variety $X$ and $U$ runs through its open subsets. Recall the following theorem of Voevodsky, a special case of his Gersten resolution \cite[Th. 4.37]{voepre}:

\begin{thm}\label{t2.1} Let $F$ be perfect and suppose $\sF\in \HI$. Then there is an exact sequence for any $X\in \Sm$:
\begin{equation}\label{eq6.1}
0\to \sF(X)\to \sF(F(X))\by{(\partial_x)} \bigoplus_{x\in X^{(1)}} \sF_{-1}(F(x))
\end{equation}
\end{thm}

Here $\partial_x$ is defined from the purity isomorphism of ibid., Lemma 4.36.

\subsection{Further characterisations of birational sheaves} The following characterisations are sometimes useful: they assume $F$ to be perfect. The first one is extracted from \cite[Lemma 10.3 b)]{somekawa}:

\begin{prop}\label{p2.2} Assume $F$  perfect, and let $\sF\in \HI$. Then $\sF\in \HI^\o$ if and only if the following holds: 
\begin{quote}
For any function field $K/F$, for any regular curve $C$ over $K$
and any closed point $c \in C$,
the induced map $\sF(\sO_{C, c}) \to \sF(K(C))$ is surjective. \qed
\end{quote}
\end{prop}

(These sheaves are called \emph{universally proper} in \cite{somekawa}.)

The second one uses the notion of contraction that we just recalled:

\begin{prop} \label{p2.3} Let $\sF\in \HI$. Consider the following conditions:
\begin{thlist}
\item $\sF\in \HI^\o$.
\item $\sF(X)\iso \sF(X\times (\A^1-\{0\}))$ for any $X\in\Sm$.
\item $\sF_{-1}=0$.
\end{thlist}
Then {\rm (i)} $\Rightarrow$ {\rm (ii)}  $\iff$ {\rm (iii)};  if $F$ is perfect, {\rm (iii)} $\Rightarrow$ {\rm (i)}.
\end{prop}

\begin{proof} The first implications are obvious; the last one follows from Theorem \ref{t2.1}.
\end{proof}

\subsection{Serre embeddings} \label{s.serre} Let $\sA$ be an abelian category. Recall that a full subcategory $\sB\subseteq \sA$ is a \emph{Serre subcategory} if it is additive and if, given an exact sequence $0\to A'\to A\to A''\to 0$ in $\sA$, we have $A\in \sB$ $\iff$ $A',A''\in \sB$. We say that $\sB\to \sA$ is a \emph{Serre embedding}. We have:

\begin{prop}\label{p1.3} let $f^*:\sB\to \sA$ be a Serre embedding. Suppose that $f^*$ has a left adjoint $f_!$ (\resp a right adjoint $f_*$). Then, for any $\sF\in \sA$ the unit (\resp counit) morphism of \eqref{equnit}
is an epimorphism (\resp a monomorphism).
\end{prop}

\begin{proof} By duality, it suffices to prove this for $f_*$. This is \cite[Prop. E.4.1 (1)]{bar-kahn}, whose proof we reproduce here for completeness. Since in particular $f^*$ is fully faithful, the unit morphism $\sG\to f_*f^*\sG$ is an isomorphism for any $\sG\in \sB$. Let $\sF\in \sA$ and let $C=\Ker(f^*f_*\sF\to \sF)$: by hypothesis, $C\simeq f^*C'$ for some $C'\in \sB$. Applying the left exact functor $f_*$ to the exact sequence $0\to C\to f^*f_*\sF\to \sF$ yields a diagram
\[\begin{CD}
0@>>> f_*f^*C'@>>> f_*f^*f_*\sF@>a>> f_*\sF \\
&&@A\wr A A @A\wr A b A\\
&& C' && f_*\sF
\end{CD}\]
in which $ab=1_{f_*\sF}$ by the adjunction identities. Hence in the top exact row, $a$ is an isomorphism and $f_*f^*C'=0$.
\end{proof}

For the sequel it is important to know that some inclusions of abelian subcategories of $\PST$ are Serre embeddings. We treat all of them in a unified way.

\begin{prop}\label{p2.4}  Suppose $F$ perfect. Then the inclusions $\iota:\PHI\allowbreak\subset \PST$, $i:\HI\subset \NST$ and $i^\o:\HI^\o\subset \HI$ are Serre embeddings.
\end{prop}

\begin{proof} Let $\sF\in \PST$. For any smooth variety $X$, the map
\[\sF(X)\to \sF(X\times \A^1)\]
is split by using the rational point $0\in \A^1$. This defines an idempotent endomorphism of $\sF(X\times \A^1)$, whose kernel we denote by $\tilde \sF(X)$: this is a presheaf in $X$. So $\sF\in \PHI$ $\iff$ $\tilde \sF=0$.

The construction $\sF\mapsto \tilde \sF$ is clearly functorial in $\sF$, and exact. In particular, if $0\to \sF'\to \sF\to \sF''\to 0$ is a short exact sequence in $\PST$ and $X\in\SmCor$, we get a short exact sequence 
\[0\to \tilde\sF'(X)\to \tilde \sF(X)\to \tilde \sF''(X)\to 0.\]

That $\iota$ is a Serre embedding follows immediately. (In this case, the perfectness of $F$ is not used.)

Suppose now that $0\to \sF'\to \sF\to \sF''\to 0$ is an exact sequence in $\NST$. Given $X\in\SmCor$, the cohomology exact sequence induces an exact sequence
\[0\to \tilde\sF'(X)\to \tilde \sF(X)\to \tilde \sF''(X)\to \tilde H^1_\Nis(X,\sF').\]

(Here we viewed $\sG=H^1_\Nis(-,\sF')$ as a presheaf on $\Sm$, and wrote $\tilde H^1_\Nis(X,\sF')$ for $\tilde \sG(X)$.) This shows that $\HI$ is closed under extensions in $\NST$. Conversely, suppose that $\sF\in \HI$. Then $\sF'\in \HI$. But a theorem of Voevodsky \cite[Th. 3.1.12]{voetri} implies that $X\mapsto H^1_\Nis(X,\sF')$ is homotopy invariant, hence $\sF''\in \HI$ as well.

To deal with the last case, we argue as above using this time $\A^1-\{0\}$. Thus, for $\sF\in \PST$, write $\check\sF(X)$ for the kernel of the idempotent on $\sF(X\times (\A^1-\{0\}))$ defined by the rational point $1\in \A^1-\{0\}$. Suppose $\sF\in \HI$: by Proposition \ref{p2.3}, $\sF\in \HI^\o$ $\iff$ $\check \sF=0$.

Suppose that $0\to \sF'\to \sF\to \sF''\to 0$ is an exact sequence in $\HI$.  For any $X\in \Sm$, we have an exact sequence
\[0\to \check\sF'(X)\to \check \sF(X)\to \check \sF''(X)\to \check H^1_\Nis(X,\sF').\]

This exact sequence shows that $\HI^\o$ is closed under extensions in $\HI$. Now suppose that $\sF\in \HI^\o$. Then $\sF'\in\HI^\o$. But lemma \ref{l14.0} b) implies that $H^1_\Nis(X,\sF')=0$, and a fortiori $\check H^1_\Nis(X,\sF')=0$. So $\sF''\in \HI^\o$ as well.
\end{proof}

Recall that in Proposition \ref{p2.4}, $i$ has a left adjoint: this nontrivial fact amounts to say that $a(\PHI)\subseteq \HI$ \cite[Prop. 3.1.12]{voetri}. More easily:

\begin{prop}\label{p5.5} In Proposition \ref{p2.4}, $i^\o$ has a left adnoint $\nu_0$ and a right adjoint $R^0_\nr$. If $F$ is perfect, the canonical morphism $\sF\to \nu_0\sF$ (\resp $R^0_\nr\sF\to \sF$) is an epimorphism (\resp a monomorphism) for any $\sF\in\HI$. 
\end{prop}

\begin{proof} Since $\HI^\o\subset \HI$, it is clear with the notation of \eqref{eq1.0} that the composition $\HI\inj \PHI\by{\beta_!}\HI^\o$ (\resp $\HI\inj \PHI\by{\beta_*}\HI^\o$) yields the desired left (\resp right) adjoint (see also Proposition \ref{padj}). The last two claims now follow from Propositions \ref{p1.3} and \ref{p2.4}.
\end{proof}

\section{Birational sheaves and pure birational motives}\label{s2}

We construct  in Proposition \ref{p2.6} a full embedding of the category $\Cor^\o_\rat$ of birational Chow correspondences, introduced in \cite{birat-pure}, into the category $\BFC$ studied in the previous section. This relies on an explicit formula for the Nisnevich sheaf with transfers $h_0^\Nis(Y)$ attached to a smooth proper variety $Y$, which imples that it is birational  (Theorem \ref{p12.3.5}); when $F$ is perfect, this was proven differently in \cite[Proof of Th. 2.2]{motiftate}. The full embedding $\Cor_\rat^\o\inj\BFC$ becomes an equivalence of categories after inverting the exponential characteristic and passing to idempotent completions (Corollary \ref{c1.1}). As a first byproduct, we get in Corollary \ref{c2.1} a functor
\[S_b^{-1}\Sm \to \Chow^\o[1/p]\]
where $S_b^{-1}\Sm$ is the category which was studied in \cite{Birat} and  $\Chow^\o$ is the category of birational Chow motives which was studied in \cite{birat-pure}; this functor could not be construced by the methods of  \cite{birat-pure}.

\subsection{Birational sheaves and smooth projective varieties} 

\begin{lemma}\label{presh} Let $X,Y\in \Sm$. Then we have a natural homomorphism (\cf \eqref{eq1.3})
\begin{equation}
\label{eq5.2}
h_0(X,Y)\to CH^{\dim
Y}(X\times Y)
\end{equation} 
which is bijective if $X=\Spec F$ and $Y$ is proper.
\end{lemma}

\begin{proof} The first claim  is obvious from the commutative diagram
\[\begin{CD}
c(X\times \mathbf{A}^1,Y)@>i_0^*-i_1^*>> c(X,Y)\\
@VVV @VVV \\
CH^{\dim Y}(X\times \mathbf{A}^1\times Y)@>i_0^*-i_1^*>> CH^{\dim Y}(X\times Y)
\end{CD}\]
and the homotopy invariance of Chow groups. Suppose that $X=\Spec F$. Then $c(X,Y)=Z_0(Y)$; if $Y$ is proper, any irreducible $1$-cycle on $\A^1\times Y$ which is not constant over $\A^1$ defines an element of $c(\A^1,Y)$, hence the second claim.
\end{proof}

Let $Y$ be a smooth variety. As in \cite[p. 207]{voetri} we write $h_0^\Nis(Y)$ for  the Nisnevich sheaf with transfers associated to the presheaf $h_0(Y)\in \PHI$ (see \eqref{eq1.0}). Note that $h_0^\Nis(Y)\in\HI$ by \cite[Th. 3.1.12]{voetri}.

\begin{thm}\label{p12.3.5} Let $Y$ be a smooth proper variety.
 Then $h_0^\Nis(Y)$ is given by the formula
\[h_0^\Nis(Y)(X)=CH_0(Y_{F(X)})\]
for any connected $X\in \Sm$. In particular, $h_0^\Nis(Y)\in \HI^\o$.
\end{thm}

\begin{proof} 
For $X$ smooth, consider the composition
\[c(X,Y)\to CH^{\dim Y}(X\times Y)\to CH^{\dim Y}(Y_{F(X)})=CH_0(Y_{F(X)}).\]

Let $\alpha\in c(Z,X)$ be a finite correspondence, with $Z$
connected. We claim that  the usual composition of correspondences
yields a commutative diagram 
\begin{equation}\label{eq14.1}
\begin{CD}
c(X,Y)@>>> CH^{\dim Y}(X\times Y)@>>> CH_0(Y_{F(X)})\\
@V{\alpha^*}V{(1)}V @V{\alpha^*}V{(2)}V @V{\alpha^*}V{(3)}V \\
c(Z,Y)@>>> CH^{\dim Y}(Z\times Y)@>>> CH_0(Y_{F(Z)}).
\end{CD}
\end{equation}

Indeed, $\alpha^*$ is well-defined at (2) because $Y$ is proper and the components of $\alpha$ are proper over $Z$. It obviously commutes with (1). 

For (3), we must show that if $W\in Z_{\dim X}(X\times Y)$ is supported on a
proper closed subset $X'$ of $X$, then $\alpha^*W$ goes to $0$ in
$CH_0(Y_{F(Z)})$. We argue as in the proof of \cite[Prop. 2.3.4]{birat-pure}: by
passing to the generic point of $Z$ and by base change, we reduce to
the case where $Z=\Spec F$. Then $\alpha$ is a $0$-cycle on $X$, that we may assume to be a closed point. Shrinking around $\alpha$, we may also assume $X$ to be quasi-projective.

If $\alpha\notin X'$,
then $\alpha\cap W=\emptyset$ and we are done; otherwise we may
move $\alpha$ outside $X'$ up to rational equivalence \cite{roberts}: this does
not change the value of $\alpha^*W$ in $CH_0(Y)$. 

Thus we have defined a presheaf with transfers 
\[\bar h_0(Y):X\mapsto CH_0(Y_{F(X)})\] 
which is clearly birationally invariant.
Hence, by Proposition 
\ref{l14.1} a), $\bar h_0(Y)\in \HI^\o$, and the morphism $L(Y)\to \bar
h_0(Y)$ described in \eqref{eq14.1} induces a morphism  
\[\phi:h_0^\Nis(Y)\to \bar h_0(Y)\] 
in $\HI$.

To prove that $\phi$ is an isomorphism, it is
sufficient by \cite[Prop. 4.20]{voepre} to check that it is an
isomorphism at $\Spec K$ for any extension $K$ of $F$. 
Noting that
\begin{gather*}
h_0^\Nis(Y)(\Spec K) =h_0(Y)(\Spec K) = h_0(Y_K)(\Spec K),\\
\bar h_0(Y)(\Spec K)=\bar h_0(Y_K)(\Spec K),
\end{gather*} 
the statement follows from Lemma \ref{presh}, where $F$ is replaced by $K$ and $Y$ by $Y_K$.\end{proof} 

A first consequence is:

\begin{cor}\label{p2.5} Let $Y$ be a smooth proper variety viewed as an object of $\BFC$. Then the associated presheaf of abelian groups $h_0^\o(Y)\in \HI^\o$ (compare Diagram \eqref{eq1.0}) is canonically isomorphic to $h_0^\Nis(Y)$. In particular, $h_0^\Nis(Y)$ is a projective object of $\HI^\o$.
\end{cor}

\begin{proof}  By definition, $h_0^\Nis(Y)$ is the Nisnevich sheaf associated to the presheaf $h_0(Y)$. Let $\sF\in \HI^\o$. By Lemma \ref{l14.0} a), $\sF$ is a Nisnevich sheaf, hence any map $h_0(Y)\to \sF$ factors uniquely through $h_0^\Nis(Y)$. Since the latter is birational by Theorem \ref{p12.3.5}, $h_0^\Nis(Y)$ has the same universal property as $h_0^\o(Y)$, so they coincide. The last statement follows from Proposition \ref{p5.3} b).
\end{proof}

\subsection{Review of some results of \cite{birat-pure}}\label{s1.1} Let $\Cor_\rat$ and $\Chow^\eff$ denote respectively the category of Chow correspondences and of effective Chow motives over $F$, with integral coefficients, so that by definition $\Chow^\eff=\Cor_\rat^\natural$. In \cite[Def. 2.2.6 and 2.3.5]{birat-pure}, we defined two new categories:
\[\Chow^\b=(\Chow^\eff/\sL)^\natural, \quad \Chow^\o=(\Cor_\rat/\sI)^\natural.\]

Here $\sL$ is the $\otimes$-ideal generated by the Lefschetz motive $\bL$ while, for two smooth projective varieties $X,Y$, $\sI(X,Y)$ is the subgroup of $CH_{\dim X}(X\times Y)$ generated by correspondences with support in $Z\times Y$ for some proper closed subset $Z\subset X$. Writing $\Cor_\rat^\o=\Cor_\rat/\sI$, Hom groups in $\Cor_\rat^\o$ are given by the formula \cite[Lemma 2.3.6]{birat-pure}
\[\Cor_\rat^\o(X,Y) = CH_0(Y_{F(X)}).\]

There is a string of full functors
\[\Chow^\b\to (S_b^{-1}\Chow^\eff)^\natural\to \Chow^\o\]
which become equivalences of categories after inverting the exponential characteristic $p$ of $F$ \cite[Th. 2.4.1]{birat-pure}. 

\enlargethispage*{20pt}

\subsection{A full embedding} We now draw other consequences from Theorem \ref{p12.3.5}. For the reader's convenience, we include a proof of the following generalisation of  Lemma \ref{presh},
which is in \cite[proof of Prop. 2.1.4]{voetri}:

\begin{thm}\label{l6.4} Let $X,Y$ be two smooth projective $F$-varieties. Then \eqref{eq5.2} is an isomorphism.
\end{thm}

\begin{proof} Let $L(Y)$ and $L^c(Y)$ be the presheaves with transfers defined in
\cite[\S 4.1]{voetri}. Then the cokernel of
$i_0^*-i_1^*$ is clearly isomorphic to
$h_0(L(Y))(X))$. On the other hand, since $Y$ is projective, the
morphism of presheaves $L(Y)\to L^c(Y)$ is an isomorphism. The latter
presheaf is canonically isomorphic to $z_\equi(Y,0)$ (compare \cite[\S
4.2]{voetri}). The group $CH^{\dim Y}(X\times Y)$, in its turn, is
canonically isomorphic to $h_0(z_\equi(X\times Y,\dim X))(\Spec
F)$. We therefore have to see
that the natural map
\[h_0(z_\equi(Y,0))(X)\to h_0(z_\equi(X\times Y,\dim X))(\Spec F)\]
is an isomorphism. But the left hand side may be further rewritten
\[h_0(z_\equi(Y,0))(X)=h_0(z_\equi(X,Y,0))(\Spec F)\]
(\cf \cite[bottom p. 142]{frivoe}). The result now follows from \cite[Th.
7.1]{frivoe}. \end{proof}

\begin{rk} Under resolution of singularities, Theorem \ref{l6.4} remains true if $X$ is only smooth quasiprojective by replacing  \cite[Th. 7.1]{frivoe} by  \cite[Th. 7.4]{frivoe} in the above proof. We shall not need this more refined result.
\end{rk}

Consider the full subcategory  $\Sm^\proj\Cor\subset \SmCor$ whose objects are smooth projective varieties, and its associated homotopy category $\sH(\Sm^\proj\Cor)\subset \sH(\SmCor)$. 
Then Theorem \ref{l6.4} yields an isomorphism of categories
\begin{equation}\label{eq2.2}
\sH(\Sm^\proj\Cor)\iso \Cor_\rat
\end{equation}
where $\Cor_\rat$ is the category of Chow correspondences (see \S \ref{s1.1}).

\begin{prop}\label{p2.6} Let $\Cor_\rat^\o$ be the category of birational Chow correspondences (see \S \ref{s1.1}). Then the identity map on objects extends to a full embedding
\[\Cor_\rat^\o\by{D} \BFC\]
which fits in the commutative diagram
\begin{equation}\label{eq2.4}
\begin{CD}
S_b^{-1}\sH(\Sm^\proj\Cor) @>A>> \BFC\\
@VCV\wr V @ADAA\\
S_b^{-1}\Cor_\rat @>B>> \Cor_\rat^\o.
\end{CD}
\end{equation}
Here $A$ is the obvious functor, $C$ is induced by \eqref{eq2.2} and $B$ is the functor from \cite[Prop. 2.3.7 c)]{birat-pure}.
\end{prop}

\begin{proof} Indeed, in view of Corollary \ref{p2.5} and \cite[Lemma 2.3.6]{birat-pure}, the isomorphism of Theorem \ref{p12.3.5} yields an isomorphism between Hom groups of the two categories $\Cor_\rat^\o$ and $\BFC$. The proof of Theorem \ref{p12.3.5} also shows that this isomorphism defines a (fully faithful) functor, and that \eqref{eq2.4} commutes if we remove $S_b^{-1}$ from the left vertical. Therefore, \eqref{eq2.4} commutes.
\end{proof}

From Propositions \ref{p5.3} a) and \ref{p2.6}, we deduce:

\begin{cor}\label{c1.1} We have a full embedding
\[\Chow^\o\inj \HI^\o\]
which sends the birational motive $h^\o(X)$ of a smooth projective variety $X$ to  $h_0^\Nis(X)$.\qed
\end{cor}

\enlargethispage*{20pt}

\subsection{More equivalences of categories}

\begin{thm}\label{t2.2} In Diagram \eqref{eq2.4}, all functors become equivalences of categories after inverting the exponential characteristic $p$ and passing to the pseudo-abelian envelopes.
\end{thm}

Thus, after inverting $p$, the categories $\Chow^\b$, $S_b^{-1}\Chow^\eff$, $\Chow^\o$ and $\BFC^\natural$ become equivalent.

\begin{proof}  
Let $A',B',C',D'$ be the corresponding functors. Then $C'$ is an isomorphism of categories by \eqref{eq2.2}, $B'$ is an isomorphism of categories by \cite[Th. 2.4.1]{birat-pure} and $D'$ is fully faithful by Proposition \ref{p2.6}. To conclude, it remains to show that $D'$ is essentially surjective.

Let $X\in \Sm$, and let $\bar X_0$ be a compactification of $X$. For each prime number $l\ne p$, choose by  \cite{gabber} an alteration $\bar X_l\to \bar X_0$ of generic degree prime to $l$, with $\bar X_l$ smooth. Choose a finite number of primes $l_1,\dots,l_r$ such that the gcd of the corresponding degrees $d_i$ is a power of $p$, say $p^s$. Choose a dense open subset $U\subseteq X$ such that $p_i:\bar X_{l_i}\to \bar X_0$ is finite over $U$ for all $i$. Let $U_i=p_i^{-1}(U)\subseteq \bar X_{l_i}$ and let $V=\coprod_{i=1}^s U_i$.

Let $\gamma_i\in c(U_i,U)$ be the graph of ${p_i}_{|U_i}$, so that its transpose ${}^t\gamma_i$ is still a finite correspondence. Choose integers $n_i$ such that $\sum n_id_i=p^s$. As $\gamma_i\circ {}^t\gamma_i = d_i 1_U$, we have
\[\sum n_i \gamma_i\circ {}^t\gamma_i = p^s 1_U.\]

Thus, if $a=\bigoplus \gamma_i\in c(V,U)[1/p]$ and $b=\frac{1}{p^s} \bigoplus n_i {}^t\gamma_i\in c(U,V)[1/p]$, then $ba$ is a projector on $V$ in $\SmCor[1/p]$, with image isomorphic to $U$.

Let $\bar X=\coprod \bar X_{l_i}\in \SmCor$. The inclusion $V\to \bar X$ becomes an isomorphism in $\BFC[1/p]$, hence the projector $ba$ yields a projector $\pi\in \End(\bar X)$ in the latter category, with image isomorphic to $U$, hence to $X$. Since $D$ is fully faithful, $\pi$ lifts to a projector in $\Cor_\rat^\o(\bar X,\bar X)[1/p]$, thereby concluding the proof. 
\end{proof}

\begin{cor}\label{c2.1} The graph functor $\Sm\to \SmCor$ induces a functor
\[S_b^{-1}\Sm =S_r^{-1}\Sm \to \Chow^\o[1/p].\qquad  \text{\qed}\]
\end{cor}

(See \cite[Th. 1.7.2]{Birat} for the equality $S_b^{-1}\Sm =S_r^{-1}\Sm$.)

\begin{rk} \label{r1.2} In \cite[Prop. 4.1]{birat}, we also proved that the functor 
\[S_b^{-1}\Sm^\proj\Cor\to  S_b^{-1}\SmCor=\BFC\]
is an equivalence of categories if $\car F=0$. Extending this result to positive characteristic (after inverting $p$ and adjoining idempotents) has defied all our attempts, even with the help of \cite[Th. 5.1.4]{localisation}. Fortunately we don't need such a result here, and leave it as a challenge for the interested readers.
\end{rk}

\section{Triangulated birational motives}

In this section, we construct a triangulated category of
birational geometric motives $\DM_\gm^\o$ that we compare with Voevodsky's category $\DM_\gm^\eff$ of \cite{voetri}. We also construct a full embedding $\BFC\inj \DM_\gm^\o$; this turns out to be much more elementary than Voevodsky's theory of (effective) triangulated motives, the main reason being that a birational presheaf is automatically a Nisnevich sheaf (Lemma \ref{l14.0} a)). We organise the exposition in order to highlight this.

In contrast to Voevodsky's approach but like Beilinson-Vologodsky \cite{be-vo}, we use unbounded derived categories in order to take advantage of Neeman's yoga of compactly generated triangulated categories \cite{neeman,neeman2},  which generalises results of Thomason-Trobaugh and Yao. We refer to \S \ref{s.term} for general notation and terminology on the latter.

\subsection{Review of effective triangulated motives} We start by recalling Voevodsky's construction of the category $\DM_\gm^\eff$. In \cite{voetri}, $\DM_\gm^\eff$ is defined as the pseudo-abelian envelope of  the Verdier quotient of $K^b(\SmCor)$ by the thick subcategory generated by the complexes of the form
\begin{description}
\item[$R_h$] $[\A^1_X]\by{[p]}[X]$, $X$ smooth;
\item[$R_{MV}$] $[U\cap V]\to [U]\oplus [V]\to [X]$, where $X$ is smooth and $U,V$ are two open subsets such that $X=U\cup V$.
\end{description}

By Balmer-Schlichting \cite{bs},  $\DM_\gm^\eff$ is  triangulated. The obvious tensor structure on $K^b(\SmCor)$ induces a tensor structure on $\DM_\gm^\eff$. 

The canonical embedding of $\SmCor$ into $K^b(\SmCor)$ sends $\SmCor$ to $\DM_\gm^\eff$; the image of $[X]$ under this functor is denoted by $M_\gm(X)$, or simply here by $M(X)$. The relations $R_h$ and the isomorphism of categories \eqref{eq2.2} yield after pseudo-abelianisation a functor
\cite[Prop. 2.1.4]{voetri} 
\begin{equation}\label{eq2}
\Chow^\eff\to \DM_\gm^\eff
\end{equation}
which sends the Chow motive $h(X)$ of a smooth projective variety $X$ to $M(X)$. 

To go further, Voevodsky constructs when $F$ is perfect a full embedding of $\DM_\gm^\eff$ into a larger triangulated category $\DM_-^\eff$ of sheaf-theoretic nature; this allows him to compute Hom groups of $\DM_\gm^\eff$ in terms of Nisnevich (or even Zariski) hypercohomology of certain complexes. As a byproduct, the functor \eqref{eq2} is fully faithful. We recall part of this story in \S \ref{s3.4}.

We shall now see that the perfectness of $F$ is not needed for the corresponding properties of triangulated birational motives.

\subsection{Triangulated birational motives} \label{s3.2}

Let us apply Proposition \ref{l8.1} with $\sA=\SmCor$. In this case, $\Mod\sA=\PST$, the category of presheaves with transfers. Thus we get that the Yoneda functor
\begin{equation}\label{eq.y}
K^b(\SmCor)\to D(\PST)
\end{equation}
is fully faithful and has dense image, whose pseudo-abelian envelope consists precisely of the compact objects of the right hand side.

\begin{defn}\label{d4.1} Let $R_\o\subset K^b(\SmCor)$ be the class of complexes $[U]\by{j}[X]$, where $j$ is an open immersion with dense image. We denote by (\cf Definition \ref{dA.2} for the notation $\langle R_\o\rangle,\langle L(R_\o)\rangle^\oplus$):
\begin{itemize}
\item $\DM_\gm^\o$ the pseudo-abelian envelope
of the Verdier quotient of
$K^b(\SmCor)$ by $\langle R_\o\rangle$. We denote the image of $[X]$ in $\DM_\gm^\o$ by
$M^\o(X)$. (It is triangulated by \cite{bs}.)
\item $\DM^\o$ the localisation of $D(\PST)$ with respect to $\langle L(R_\o)\rangle^\oplus$. We define $\DM_-^\o$ similarly, using $D^-(\PST)$ instead of $D(\PST)$.
\end{itemize}
\end{defn}

For the next theorem, recall the notation
\begin{equation}\label{eq3.4}
D_\sB(\sA)=\{C\in D(\sA)\mid H^i(C)\in \sB\quad \forall i\in\Z\}
\end{equation}
if $\sB$ is a strictly full subcategory of an abelian category $\sA$: this is a triangulated subcategory of $D(\sA)$ provided $\sB$ is \emph{thick} in $\sA$, \ie given a short exact sequence in $\sA$, if two terms belong to $\sB$ then so does the third. Note that $\HI^\o$ is thick in $\PST$.

\begin{thm}\label{t3.0} Let $\gamma:\SmCor\to \BFC$ be the localisation functor. \\
 a) The functor $K^b(\gamma)^\natural:K^b(\SmCor)^\natural\to K^b(\BFC)^\natural$ factors through $\DM_\gm^\o$. The total derived functor $L\gamma_!:D(\PST)\to D(\HI^\o)$ factors through $\DM^\o$. This yields a naturally commutative diagram
\[\xymatrix{
&\SmCor\ar[d]_\gamma\ar[r]^{\eta} &K^b(\SmCor)^\natural\ar[d]_{\bar \gamma}\ar[r]^{\iota}& D(\PST)\ar[d]_{\bar \gamma_!}\\
\Chow^\o\ar[r]^{D^\natural}&\BFC^\natural\ar[r]^{\eta^\o}\ar[rd]^{\eta'}&\DM_\gm^\o\ar[r]^{\iota^\o}\ar[d]^w& \DM^\o\ar[d]^{w_\oplus}\\
&&K^b(\BFC)^\natural\ar[r]^{\iota'}&D(\HI^\o)
}\]
in which all functors not starting from $\Chow^\o$, $\SmCor$ or $\BFC^\natural$ are triangulated.\\
b) The functors $\eta,\eta^\o$ and $\eta'$ are fully faithful. The functors $\iota,\iota^\o$ and $\iota'$ are fully faithful with dense images, and identify their domains with the full subcategory of compact objects of their range.\\
c) The functor $\bar \gamma_!$ has a (fully faithful) right adjoint $\bar \gamma^*$, which itself has a right adjoint $\bar \gamma_*$. The essential image of $\bar \gamma^*$ is $D_{\HI^\o}(\PST)$, where $\HI^\o$ is embedded in $\PST$ by means of $\gamma^*$.\\
d) Via $\bar \gamma^*$, the natural $t$-structure of $D(\PST)$ induces a $t$-structure on $\DM^\o$, with heart $\HI^\o$; the functor $\bar \gamma_!$ (\resp $\bar \gamma_*$) is  right (\resp left) $t$-exact.\\
e) For $X,Y\in \Sm$ and $\sF\in \HI^\o$, we have: 
\begin{align*}
\DM^\o(M^\o(X),\sF[q])&=
\begin{cases}
0& \text{for $q\ne 0$}\\
\sF(X)&\text{for $q=0$},
\end{cases}\\
\DM^\o(M^\o(X),M^\o(Y)[i])&=0 \quad \text{for } i>0\\
\sH^i(\bar\gamma^* M^\o(Y))&=0 \quad \text{for } i>0.
\end{align*}
f) If  $Y$ is proper, we have 
\begin{align*}
\DM^\o(M^\o(X),M^\o(Y))&=CH_0(Y_{F(X)})\\
\sH^0(\bar\gamma^* M^\o(Y))&=h_0^\Nis(Y).
\end{align*}
The functor $D^\natural$ is fully faithful (hence so is $\eta^\o D^\natural$).\\
g) The obvious functor
\[\phi:D(\HI^\o)\to D_{\HI^\o}(\PST)\simeq \DM^\o\]
is right adjoint to $w_\oplus$, $t$-exact and conservative; it induces the identity on the hearts.\\
h) The functor $w_\oplus$ is right $t$-exact and induces the identity on the hearts; its restriction to $\DM^\o_-$ is conservative. The functor $w$ is conservative as well.
\end{thm}

\begin{proof} Everything follows directly from Proposition \ref{l8.1} and Theorem \ref{tC.1}, except for f) which follows from Theorem \ref{p12.3.5}, Corollary \ref{p2.5} and Proposition \ref{p2.6}. 
\end{proof}

\enlargethispage*{20pt}

\begin{rk} By theorem \ref{t2.2}, $D^\natural$ becomes essentially surjective after inverting the exponential characteristic $p$.
\end{rk}

To Theorem \ref{t3.0}, we add:

\begin{prop}\label{p3.2} The $\otimes$-structure on $\SmCor$ induces a $\otimes$-structure on all categories in the diagram of Theorem \ref{t3.0}, and all functors in this diagram are $\otimes$-functors. The $\otimes$-structures are compatible with the triangulated structures when applicable.
\end{prop}

\begin{proof} Indeed, if $U\by{j}X$ is a dense open immersion, then $U\times Y\by{j\times 1_Y} X\times Y$ is also a dense open immersion for any $Y\in \Sm$.
\end{proof}

The next result is deeper:

\begin{prop}\label{p5.1}   The thick subcategory $\langle R^\o\rangle^\natural \subset \DM_\gm^\eff$ contains all motives of the form
$M(1):=M\otimes \Z(1)$. If $F$ is perfect, this is an equality and the functor $\DM_\gm^\eff\to \langle R^\o\rangle^\natural$ given by $M\mapsto M(1)$
is an equivalence of categories. Similarly, $\DM^\eff(1)\subseteq \langle R^\o\rangle^\oplus$ with equality when $F$ is perfect.
\end{prop}

\begin{proof} By density, the case of $\langle R^\o\rangle^\oplus$ reduces to that of $\langle R^\o\rangle^\natural$.
  Let $\sD$ be the full subcategory of $\DM_\gm^\eff$ consisting of the motives of the form $M(1)$. Since
$\Z\oplus \Z(1)[2]=M^\o(\P^1)=M^\o(\A^1)=\Z$ in $\DM_\gm^\o$, $\Z(1)=0$ in
$\DM_\gm^\o$. Therefore   $\sD\subseteq \langle R^\o\rangle^\natural$ by Proposition \ref{p3.2}. 

To see the converse inclusion when $F$ is perfect, we first prove that $\sD$ is a triangulated subcategory of $\DM_\gm^\eff$. We have to show that, if $M,N\in \DM_\gm^\eff$ and $f\in
\Hom(M(1),N(1))$, then the cone of $f$ is of the form $P(1)$. This follows
from the cancellation theorem of \cite{voecan}.  The
cancellation theorem also shows that $M\mapsto M(1)$ yields an equivalence of categories $\DM_\gm^\eff\iso \sD$.

We now have to prove that $M(U)\by{j_*}M(X)$ is an isomorphism in $\DM_\gm^\eff/\sD$ for any open immersion
$j$. We argue by Noetherian induction on the (reduced) closed complement
$Z$ in a standard way. For simplicity, let us say that the open immersion $j$ is \emph{pure} if
$Z$ is smooth. If $j$ is pure, then the cone of $j_*$ is isomorphic to
$M(Z)(c)[2c]$, where $c=\codim_X Z$ by the Gysin exact triangle of \cite[Prop. 3.5.4]{voetri}, so the claim is true in this case. In general, filtering $Z$ by its successive singular loci, we may write $j$ as a composition of pure open immersions, and the claim follows.
\end{proof}

\enlargethispage*{20pt}

\subsection{Relationship between effective and birational triangulated motives}\label{s3.3}  We already introduced three classes $R_h,R_{MV}$ and $R_\o$ of objects of $K^b(\SmCor)$. Here we shall use two others:
\begin{itemize}
\item $R_\Nis=\{[B]\to [A]\oplus [Y]\to [X] \}$, where
\begin{equation}\label{eq3.3}
\begin{CD}
B@>j'>> Y\\
@V{\pi'}VV @V{\pi}VV \\
A @>j>> X
\end{CD}
\end{equation}
is an upper-distinguished square in the sense of \cite[Def. 12.5]{mvw} (also called elementary distinguished square in \cite[p. 96, Def. 1.3]{mv}).
\item $R=R_h\cup R_\Nis$.
\end{itemize}

\begin{lemma}\label{l2.1} We have the following inclusions:
\begin{thlist}
\item $R_{MV}\subset R_\Nis$;
\item $\langle R_\Nis\rangle  \subset \langle R_\o\rangle$;
\item $\langle R_h\rangle  \subset \langle R_\o\rangle$;
\item $\langle R\rangle \subset \langle R_\o\rangle$.
\end{thlist}
Moreover, the classes $R_h,R_{MV},R_\Nis,R$ and $R_\o$ are stable under $-\otimes [X]$ for any $X\in \Sm$.
\end{lemma}

\begin{proof} This is essentially a reformulation of Lemmas \ref{l14.0} and \ref{l14.1}. (i) is obvious. (ii) follows from the fact that, in \eqref{eq3.3}, $j$ and $j'$ are open immersions. (iii) follows from \cite[Th. 1.7.2]{Birat}. (iv) follows from (ii) and (iii). The last statement is obvious (and already observed for $R_\o$).
\end{proof}

We also recall the following important fact  \cite[p. 1749, (4.3.1)]{be-vo}:

\begin{prop}\label{p3.1} Let $La:D(\PST)\to D(\NST)$ be the  functor  induced by the exact functor $a$ of \eqref{eq3.2}. Then $La$ is a localisation, with kernel $\langle R_\Nis\rangle^\oplus$. \qed
\end{prop}

It is now convenient for the exposition to introduce the ``Nisnevich competitor" of $\DM_\gm^\eff$:

\begin{defn} The category $(\DM_\gm^\eff)_\Nis$ is the pseudo-abelian envelope of the Verdier quotient $K^b(\SmCor)/\langle R\rangle$. We also define $\DM^\eff=D(\PST)/\langle L(R)\rangle^\oplus$ and $\DM_-^\eff=D^-(\PST)/\langle L(R)\rangle^\oplus$.
\end{defn}

From Proposition \ref{p3.1}, we deduce:

\begin{prop} The categories $\DM^\eff$ and $\DM_-^\eff$ are respectively equivalent to $D(\NST)/\langle L(R_h)\rangle^\oplus$ and $D^-(\NST)/\langle L(R_h)\rangle^\oplus$. In particular, $\DM_-^\eff$ coincides with the category defined in \cite[Def. 14.1]{mvw}.\nobreak\qed
\end{prop}

Proposition \ref{p3.1} and Lemma \ref{l2.1} yield the following naturally commutative diagram, which summarises what we got so far (the notations $LC$ and $\nu_{\le 0}$ will be explained in \S \ref{sadj}\footnote{In \cite[Prop. 3.2.3]{voetri}, the bounded above version of the first functor is denoted by $\mathbf{R}C$, but we prefer the notation $LC$ as it is a left adjoint.
}):
\begin{equation}\label{p2.1}
\xymatrix{
&\SmCor\ar@{^{(}->}[r]\ar[rd]\ar[dd]&K^b(\SmCor)^\natural \ar@{^{(}->}[r]^c\ar@{->>}[d]^{\natural}& D(\PST)\ar@{->>}[d]^{La}\\
&&\displaystyle\left(\frac{K^b(\SmCor)}{\langle R_\Nis\rangle}\right)^\natural \ar@{^{(}->}[r]^{\qquad c}\ar@{->>}[d]^\natural&D(\NST)\ar@{->>}[dd]^{LC}\\
\Chow^\eff\ar[dd]\ar@{^{(}->}[r]&\sH(\SmCor)\ar[r]\ar[dr]\ar[dd]&\DM_\gm^\eff\ar@{->>}[d]^\natural\ar[rd]\\
&&(\DM_\gm^\eff)_\Nis\ar@{^{(}->}[r]^c\ar@{->>}[d]^\natural& \DM^\eff\ar@{->>}[d]^{\nu_{\le 0}}\\
\Chow^\o\ar@{^{(}->}[r]&\BFC^\natural\ar@{^{(}->}[r]&\DM_\gm^\o\ar@{^{(}->}[r]^c& \DM^\o.
}
\end{equation}

In \eqref{p2.1}, all categories are $\otimes$ categories and all functors are $\otimes$-functors. In the two right columns, the categories and functors are triangulated. We use $\Inj$ (\resp $\Surj,\overset{\natural}{\Surj}$) to denote a full embedding (\resp a localisation, a localisation followed by taking pseudo-abelian envelope), and the letter $c$ means that the corresponding functor is a dense embedding of the full subcategory of compact objects (see \S \ref{s.term}).

The full and dense embeddings on the top and bottom rows come from Theorem \ref{t3.0}, while the two other full and dense embeddings follow from applying Theorem \ref{neem2}.

We also have:

\begin{thm} \label{c3.1} The functors $La$, $LC$ and $\nu_{\le 0}$ of Diagram \ref{p2.1} have (fully faithful) right adjoints $Rk,i$ and $i^\o$, which in turn have right adjoints. 
\end{thm}

\begin{proof} For $La$ and the compositions $LC La,\nu_{\le 0} LC La$, this follows from the dense embedding \eqref{eq.y} and Theorem \ref{neem3}. For the individual functors, we now get the adjoints from Proposition \ref{padj} below.
\end{proof}

\begin{prop}\label{padj} Let 
\[\sC\by{F}\sD\by{G}\sE\]
be a sequence of categories and functors; let $H=GF$.\\
a) Suppose  that $H$ has a right adjoint $H_*$ and that  $G$ is fully faithful. Then $F$ has a right adjoint, given by $F_*=H_*G$.\\
b) Suppose  that $H$ and $F$ have right adjoints $H^*$ and $F^*$ and that  $F$ is a localisation. Then $G$ has a right adjoint, given by $G^*=FH^*$.
\end{prop}

\begin{proof} a) For $c\in \sC$ and $d\in \sD$, we have a map
\[\sD(Fc,d)\by{G} \sE(GFc,Gd)\iso \sC(c,H_*Gd) \]
which is natural in $d$ and $c$, and $G:\sD(Fc,d)\to \sE(GFc,Gd)$ is bijective since $G$ is fully faithful.

b) The argument is similar but a little more delicate: let $d\in \sD$ and $e\in \sE$. Since $F$ is a localisation, it is surjective so that $d= Fc$ for some $c\in \sC$. We then have a map
\[\sE(Gd,e)= \sE(GFc,e)\iso \sC(c,H^*e)\by{F} \sD(Fc,FH^*e) = \sD(d,FH^*e).\]

By adjunction, the map $F:\sC(c,H^*e)\to \sD(Fc,FH^*e)$ is converted into the map
$\sC(c,H^* e)\to \sC(c,F^*FH^*e)$ induced by the unit morphism $H^*e\to F^*FH^*e$. Let us show that the latter is an isomorphism. Since $F$ is a localisation, $F^*$ is fully faithful \cite[I.1.4]{gz} and it suffices to see that $H^*e\in \IM F^*$, which is true since
\begin{multline*}
\IM F^*=\{\gamma\in \sC\mid \sC(s,\gamma) \text{ is bijective for all } s\\
 \text{ such that } F(s) \text{ is invertible}\}\supseteq \IM H^*.
\end{multline*} 
\end{proof}

\begin{rks} 1) Passing to the categories of presheaves, one can see that the existence of $H^*$ is not necessary in the hypothesis of Proposition \ref{padj} b).

2) Using standard arguments for unbounded triangulated categories \cite{spaltenstein}, one sees that the right adjoint $Rk$ of $La$ is the total derived functor of $k:\NST\inj \PST$.
\end{rks}

\subsection{The case of a perfect field}\label{s3.4} We have:

\begin{thm}\label{t3.2} Suppose $F$ perfect. Then, in Diagram \eqref{p2.1}:
\begin{enumerate}
\item The composite functor $\DM_\gm^\eff\to (\DM_\gm^\eff)_\Nis\to \DM^\eff$ is fully faithful.
\item The functor $\DM_\gm^\eff\to (\DM_\gm^\eff)_\Nis$ is an equivalence of categories.
\item The functor $\Chow^\eff\to \DM_\gm^\eff$ is fully faithful.
\end{enumerate}
All this is summarised in the following simpler diagram:
\begin{equation}\label{p2.1perf}
\xymatrix{
&K^b(\SmCor)^\natural \ar@{^{(}->}[r]^c\ar@{->>}[d]^{\natural}& D(\PST)\ar@{->>}[d]^{La}\\
&\displaystyle\left(\frac{K^b(\SmCor)}{\langle R_\Nis\rangle}\right)^\natural \ar@{^{(}->}[r]^{\qquad c}\ar@{->>}[d]^\natural&D(\NST)\ar@{->>}[d]^{LC}\\
\Chow^\eff\ar[d]\ar@{^{(}->}[r]&\DM_\gm^\eff\ar@{->>}[d]^\natural\ar@{^{(}->}[r]^c &\DM^\eff\ar@{->>}[d]^{\nu_{\le 0}}\\
\Chow^\o\ar@{^{(}->}[r]&\DM_\gm^\o\ar@{^{(}->}[r]^c& \DM^\o.
}
\end{equation}
Finally, the canonical $t$-structure of $D(\NST)$ induces a $t$-structure with heart $\HI$ on $\DM^\eff$ via the right adjoint $i$ to $LC$; the latter induces on $\DM^\o$ the $t$-structure of Theorem \ref{t3.0} d) via the right adjoint $i^\o$ to $\nu_{\le 0}$.
\end{thm}

\begin{proof} This summarises some of the main results of Voevodsky. Namely, (1) is \cite[Th. 3.2.6 1]{voetri}, and (2) follows from (1) since the first (\resp second) functor in (1) is a localisation (\resp is fully faithful) by Diagram \eqref{p2.1}. (3) is proven  in \cite[Cor. 4.2.6]{voetri} and \cite[Prop. 20.1]{mvw} under resolution of singularities; in \cite[Cor. 6.7.3]{be-vo}, this is extended to any perfect field by a simple duality argument. As for $t$-structures, the first statement is \cite[comment after Prop. 3.1.13]{voetri}\footnote{Recall that all the above relies on the highly nontrivial fact that a homotopy invariant Nisnevich sheaf with transfers is strictly homotopy invariant: \cite[Th. 3.1.12 1]{voetri} or \cite[Th. 13.8]{mvw}.}. The second one amounts to say that $i^\o$ is $t$-exact. This follows from Theorem \ref{t3.0} d), namely the $t$-exactness of $\bar\gamma^*=Rkii^\o$, and from the exactness of the sheafification functor $a:\PST\to \NST$. Namely, let $C\in (\DM^\o)^{\le 0}$. By Theorem \ref{t3.0} d), $\sH^i(Rkii^\o C)=0$ for $i>0$, hence $\sH^i(ii^\o C)= a\sH^i(Rkii^\o C)=0$ for $i>0$ as well, and $i^\o C\in (\DM^\eff)^{\le 0}$. The reasoning is the same to get $i^\o (\DM^\o)^{\ge 0}\subset (\DM^\eff)^{\ge 0}$.
\end{proof}

\begin{rks} \label{l4.2} 1) Consider the $t$-structures of Theorem \ref{t3.2}. The right adjoint $ii^\o$ of $\nu_{\le 0}LC$ is $t$-exact. On the other hand, $a:D(\PST)\to D(\NST)$ is $t$-exact but its right adjoint $Rk$ is clearly not $t$-exact. Neither is the composition $Rk\circ i$ when $F$ is perfect: for example, $\sH^1(Rk\circ i(\bG_m))$ is the presheaf $X\mapsto \Pic(X)$. However, the composition of all right adjoints $Rk\circ i\circ i^\o$
is $t$-exact, as just used in the above proof.\\
2) In \cite[Th. 4.3.3]{birat-pure}, we showed that the functor $\Chow^\eff\to \Chow^\o$ does not have a right adjoint, even after tensoring Hom groups with $\Q$; more precisely, this right adjoint is not defined at the motive of a suitable smooth projective 3-fold. We shall strengthen this result in Remark \ref{r5.4} by showing that the right adjoint of  $\DM_\gm^\eff\to \DM_\gm^\o$ is not defined at the motive of a suitable smooth projective 3-fold, even after tensoring with $\Q$.
\end{rks}

\subsection{The essential images of $i$ and $i^\o$}\label{pnst}

The following proposition computes some Hom groups in $D(\NST),\DM^\eff$ and $\DM^\o$:

\begin{prop}\label{p4.3} Let $X$ be a smooth scheme over $F$ and let $C$ be an object of $D(\NST)$ (\resp $\DM^\eff,\DM^\o$). Then there is a canonical isomorphism
\begin{align*}
D(\NST)(L(X),C)&\simeq H^0_\Nis(X,C)\\
(\resp \DM^\eff(M(X),C)&\simeq H^0_\Nis(X,C),\\
\DM^\o(M^\o(X),C)&\simeq H^0_\Nis(X,C)).
\end{align*}
\end{prop}

(Note that we write $H^i_\Nis(X,C)$ for the Nisnevich hypercohomology of $C$, 
which is sometimes written $\bH^i_\Nis(X,C)$.)

\begin{proof} This is \cite[Prop. 3.1.8]{voetri} when $C$ is bounded above; but the same
argument works for an unbounded $C$ by replacing an injective resolution of $C$ by a
$K$-injective resolution in the sense of Spaltenstein \cite[Th. 4.5 and Rk. 4.6]{spaltenstein}, compare \cite[Ex. 13.5]{mvw}. The other statements follow by adjunction.
\end{proof}

We shall also need:

\begin{prop}\label{p4.4}
 The internal Hom of $D(\NST)$ induces an internal Hom on $\DM^\eff$ via $i$. We denote it by $\uHom_\eff$.
\end{prop}

\begin{proof} This follows by adjunction from the fact that $LC:D(\NST)\to \DM^\eff$ is a $\otimes$-functor, as observed just below \eqref{p2.1}.
\end{proof}

\begin{cor}\label{p3.3} The essential image of  $i:\DM^\eff\to D(\NST)$ (\resp $i^\o:\DM^\o\to \DM^\eff$) is the full subcategory
of those complexes $C$ such that
\[H^*_\Nis(X,C)\iso H^*_\Nis(X\times \A^1,C)\]
for all smooth $X$ (\resp such that
\[H^*_\Nis(X,C)\iso H^*_\Nis(U,C)\]
for any dense open immersion $U\to X$ of smooth varieties).
\end{cor}

\begin{proof} This follows from Proposition \ref{p4.3} and Theorem \ref{neem3} (ii). 
\end{proof}

Here is an alternate description of $i^\o\DM^\o$.
The following lemma follows from Proposition
\ref{p5.1} and Theorem \ref{neem3}:

\begin{lemma}\label{l6.1} $i^\o \DM^\o\subseteq \{C\in \DM^\eff\mid
\uHom_\eff(\Z(1),C)=0\}$ (see Proposition \ref{p4.4}). If $F$ is perfect, this inclusion is an equality.\qed
\end{lemma}

As was already observed in \cite{motiftate} and \cite{kl}, this implies that the terms of the ``associated graded of the slice filtration'' on an object of $\DM^\eff$ are twists of birational motives.

\subsection{Computing $i^\o\nu_{\le 0}$}\label{sadj} In this subsection, we assume $F$ perfect. We first recall Voevodsky's computation of $i LC$ in this case. Recall that, if $\sF\in \NST$, the \emph{Suslin complex} $\uC_*(\sF)$ of $\sF$ is the (chain) complex of Nisnevich sheaves with transfers given in degree $n$ by
\[C_n(\sF)(X)=\sF(X\times \Delta^n)\]
where the differentials are induced by linear combinations of the face maps \cite[p. 207]{voetri}, \cite[Def. 2.12]{mvw}. If $K$ is a bounded below chain complex (= bounded above cochain complex) of Nisnevich sheaves with transfers, we can extend this definition by
\[\uC_*(K)=\Tot (p\mapsto \uC_*(K_p)).\]

Finally, if $K\in C(\NST)$, we define
\[\uC_*(K) =\hocolim \Tot \uC_*(\tau_{\le n} K)\]
(see \cite{bn} for $\hocolim$).

This defines an endofunctor of $D(\NST)$. Then:

\begin{prop}[\protect{\cite[Prop. 3.2.3]{voetri}}] \label{p4.2} For any $K\in D(\NST)$, we have a natural isomorphism
\[i LC(K)\simeq \uC_*(K).\]
\end{prop}

 We now study the functor $\nu_{\le 0}$ along with its right adjoint $i^\o$. In this case, the story is much simpler.

Consider the inclusion functor $\DM^\eff(1)\inj \DM^\eff$. Using the cancellation theorem, its right adjoint is trivial to  write down: it is given by
\[M\mapsto \uHom_\eff(\Z(1),M)(1)\]
where $\uHom_\eff$ is the internal Hom in $\DM^\eff$ (see Proposition \ref{p4.4}). From Proposition \ref{p5.1} and Theorem \ref{neem3} (iii), we then immediately get a formula for $i^\o\nu_{\le 0}$:

\begin{prop}\label{l5.3}If $F$ is perfect,  we have an exact triangle for any $M\in \DM^\eff$
\[\uHom_\eff(\Z(1),M)(1)\to M\to i^\o\nu_{\le 0}M\by{+1}.\]
\end{prop}

In \cite{motiftate}, this appears as part of the description of the slice filtration on $M$.

\begin{rk} In \cite[Rk. 2.2.6]{kl} there is a different ``computation" of the functor $i^\o\nu_{\le 0}$, in the spirit of Proposition \ref{p4.2}: for a smooth $F$-variety $X$ with function field $K$, write $\hat \Delta^n_X = \hat \Delta^n_K$ for the semi-localisation of $\Delta^n_K$ at the vertices. This defines a sub-cosimplicial scheme of $\Delta^*_K$ and $\Delta^*_X$. Thus, for $\sF\in \PST$, we may define 
\[\hat C_*(\sF)(X) = \sF(\hat\Delta^*_X);\]
the $n$-th term of this chain complex is $\hat C_n(\sF)(X) = \sF(\hat \Delta^n_X)$ (defined by an inductive limit), and the differentials are induced by the face maps. We may then extend $\hat C_*$ to $C(\PST)$ as above. Then, for $N\in C(\PST)$ with homotopy invariant homology presheaves, we have a canonical isomorphism
\[i^\o\nu_{\le 0} La N(X) \simeq \hat C_*(N)\in D(\Ab).\]
\end{rk}

\begin{rk}\label{r5.2} Theorem \ref{t3.0} e) and f) identifies some of the homology (pre)sheaves $h_q^\o(Y)$ of $M^\o(Y)$ for a smooth proper variety $Y$: they are $0$ for $q<0$ and $h_0^\o(Y)=h_0^\Nis(Y)$. One may wonder about $q>0$.  If $Y$ is a curve, we have $h_q^\Nis(Y)=0$ for $q>1$ and $h_1^\Nis(Y)=p_*p^*\bG_m$ by \cite[Th.
3.4.2]{voetri}, 
where $p$ is the structural morphism. Thus we get an exact triangle
\[(p_*\Z)(1)[2]\to M(Y)\to h_0^\Nis(Y)[0]\by{+1}\]
whence $M^\o(Y)\iso h_0^\Nis(Y)[0]$ by Corollary \ref{p2.5}.  In this case, we therefore find $h_q^\o(Y)=0$ for all $q\ne 0$.

This is no longer true when $Y$ is a surface. Indeed,  at least if $F$ is algebraically closed, 
one can then produce an isomorphism
\[h_1^\o(Y)\simeq H^3_\ind(Y,\Z(2))\]
where the right hand side is the quotient of $H^3(Y,\Z(2))\simeq H^1(Y,\sK_2)$ by the image of $\Pic(Y)\otimes F^*$. See \cite[Th. 4.1]{rqnr}.
\end{rk}

\begin{rk}\label{r5.4} We can now justify Remark \ref{l4.2} 2): let $M\in \DM_\gm^\o$. By the universal property of a right adjoint and the full faithfulness of $\DM_\gm^\o\to \DM^\o$, the right adjoint of $\DM_\gm^\eff\to \DM_\gm^\o$ is defined at $M$ if and only if $i^\o\nu_{\le 0}M\in \DM_\gm^\eff$ (and then $i^\o\nu_{\le 0}M$ is the value of this right adjoint). Suppose $F$ perfect. By the exact triangle of Proposition \ref{l5.3}, the latter is equivalent to $\uHom_\eff(\Z(1),M)(1)\in \DM_\gm^\eff$.

Let us show that this fails for $M=M^\o(X)$, $X$ a suitable $3$-dimen\-sion\-al smooth projective variety.\footnote{Recall that the right adjoint is defined at $M^\o(X)$ for $X$ smooth projective of dimension $\le 2$, at least after tensoring Hom groups with $\Q$, by \cite[Th. 7.8.4 b)]{kmp}.} In \cite[App. A]{huber-ayoub},  Ayoub proved that if the Griffiths group of $X$ is not finitely generated, then $\uHom_\eff(\Z(1),M(X))$ is not compact in $\DM^\eff$, hence is not in $\DM_\gm^\eff$; this works even with $\Q$-coefficients. (There are several examples of such $X$, starting from a general quintic hypersurface in $\P^4$  over $F=\C$ which is the original example of Clemens and Griffiths.) It suffices to show that the same then holds for $\uHom_\eff(\Z(1),M)(1)$: but this is clear by the cancellation theorem \cite{voecan},  since $K\mapsto K(1)$ commutes with infinite direct sums. 
\end{rk}

\subsection{$t$-structures and projective objects} (Compare \protect{\cite[p. 1737, Footnote 17]{be-vo}}.)\label{r5.1} 

Let $\sS$ be a triangulated category with a $t$-structure, with heart $\sA$. Let us say that an object $S\in \sS$ is \emph{projective} (with respect to the $t$-structure) if $\sS(S,A[i])=0$ for all $A\in \sA$ and $i\ne 0$. Theorem \ref{t3.0} e), and more generally Theorem \ref{tC.1} e),  gives such examples. If $S$ is bounded above (i.e. $S\in \sS^{\le n}$ for some $n$), an inductive Yoneda-style argument shows that $S\in \sS^{\le 0}$. Moreover, for any $A\in \sA$, one has
\[\sS(H^0(S),A)\iso \sS(S,A),\qquad \sS(H^0(S),A[1])=0.\]

The second equality implies that $\Ext^1_\sA(H^0(S),A)=0$, hence $H^0(S)$ is a projective object in $\sA$.

Note that, then, $\Ext^i_\sA(H^0(S),A)=0$ for all $i>0$. However, if $\sS(H^0(S),A[i])=0$ for all $i>0$ and the $t$-structure is non-degenerate, the same Yoneda argument shows that $S\iso H^0(S)[0]$.

This applies to show that the natural functor $D(\HI)\to \DM^\eff$ is not full: take $S=\Z(2)[2]$, use Proposition \ref{p4.6} below to see that $S$ is projective, and then the fact that $\sH^{-1}(S)\simeq \sK_3^\ind\ne 0$ \cite{levineind,msind} to get a contradiction.

By Remark \ref{r5.2}, the isomorphism $S\iso H^0(S)[0]$ also fails in general for $\sS=\DM^\o$ and $S=M^\o(X)$, $X$ smooth proper. In particular, the conservative functor $D(\HI^\o)\to \DM^\o$ of Theorem \ref{t3.0} g) is not full.

\section{Further examples of birational sheaves}

Throughout this section, $F$ is supposed perfect.

\subsection{Constant sheaves, abelian varieties and $0$-cycles} All these are examples of birational sheaves. For the first ones, this is obvious. If $A$ is an abelian variety, then $X\mapsto A(X)$ defines an object of $\HI$ \cite[Lemma 1.3.2]{bar-kahn}, and this sheaf is birational by \cite[Th. 3.1]{milne}. Finally, for any smooth proper variety $Y$, the assignment $X\mapsto CH_0(Y_{F(X)})$ defines an object of $\HI^\o$ by Theorem \ref{p12.3.5}.

\subsection{Birational sheaves and contractions}  Recall:

\begin{prop}[\protect{\cite[Prop. 4.3 and Rk. 4.4]{somekawa}}]\label{p4.6}  The exact endofunctor $\sF\mapsto \sF_{-1}$ of $\HI$ (see \S \ref{s.contr}) is given by the formulas
\[\sF_{-1} = \sH^0\left(\uHom_\eff(\Z(1)[1],\sF[0])\right) = \uHom_{\HI}(\bG_m,\sF)\]
where $\uHom_{\HI}$ is the internal $\Hom$ of the category $\HI$. Moreover, $\uHom_\eff(\Z(1)[1],\sF[0])$ is acyclic in degrees $\ne 0$.
\end{prop}

By this proposition, one has  canonical isomorphisms for any $C\in \DM^\eff$
\begin{equation}\label{eq4.4}
\sH_n(\uHom_\eff(\Z(1)[1],C))\simeq \sH_{n}(C)_{-1}\qquad (n\in\Z).
\end{equation}

The next proposition gives some handle on the functor
$i^\o\nu_{\le 0}$. We shall use the following notation: for $\sF,\sG\in \HI$, we write 
\begin{equation}\label{eqTor}
\Tor_i^{\DM}(\sF,\sG)= \sH_i(\sF[0]\otimes \sG[0]), \quad \sF\otimes_{\HI} \sG= \Tor_0^{\DM}(\sF,\sG)
\end{equation}
where the tensor product $\sF[0]\otimes \sG[0]$ is computed in $\DM^\eff$. 

Note that  $\otimes_{\HI}$ is the tensor product on $\HI$ induced by the one on $\DM^\eff$; it is right $t$-exact because the tensor product in $D(\PST)$ is right $t$-exact and the functor $LC\circ La:D(\PST)\to \DM^\eff$ is right $t$-exact. In particular, $\Tor_i^{\DM}(\sF,\sG)=0$ for $i<0$. On the other hand, $\Tor_i^{\DM}(-,-)$ need not yield the $i$-th derived functor of $\otimes_{\HI}$  (assuming it exists), see \S \ref{r5.1}.

Let $\sF\in \HI$. By adjunction, Proposition \ref{p4.6} yields a map
\begin{equation}\label{eq5.3}
\epsilon_\sF:\bG_m\otimes_{\HI}\sF_{-1}\to \sF.
\end{equation}

\begin{prop}\label{nu0C} 
 a) Let $C\in \DM^\eff$ be such that $\sH_q(C)=0$ for $q<0$. Then $\sH_0(i^\o\nu_{\le 0} C)=\Coker\epsilon_{\sH_0(C)}$.\\
 b) If $C$ is
a sheaf $\sF$ concentrated in degree $0$, we have
\[\sH_{q}(i^\o\nu_{\le 0}\sF[0])=\begin{cases}
\Tor_{q-1}^{\DM}(\bG_m,\sF_{-1})&\text{for $q>1$} \\
\Ker \epsilon_\sF&\text{for $q=1$}\\
\Coker\epsilon_\sF&\text{for $q=0$}\\
0 &\text{for $q<0$.}
\end{cases}\]
In particular, the sheaves on the right hand side belong to $\HI^\o$. 
\end{prop}

\begin{proof}  
Using Proposition \ref{l5.3}, \eqref{eq4.4} gives a) for $n=0$ and b) if $C=\sF[0]$.
\end{proof}

We can strengthen this proposition as follows, thus giving an
interesting way to produce birational sheaves:

\begin{prop}\label{p14.1} Let $\sF\in \HI$. Then\\
a) For all $q>0$, $\Tor_q^{\DM}(\bG_m,\sF)\in \HI^\o$.\\
b) The adjunction map
\[\sF\to (\bG_m\otimes_{\HI}\sF)_{-1}\]
is an isomorphism.
\end{prop}

\begin{proof} 
We shall prove a) and b) together. By quasi-invertibility of $\Z(1)$ \cite{voecan}, the
adjunction map
\[\sF\to \uHom_\eff(\Z(1),\sF(1))\]
is an isomorphism. The right hand side may be rewritten
$\uHom_\eff(\bG_m,\bG_m\allowbreak\otimes\sF)$. By
Proposition \ref{p4.6}, its $q$-th homology sheaf is
$\Tor^{\DM}_q(\bG_m,\sF)_{-1}$, which is therefore $0$ for $q>0$. The conclusion now  follows from Proposition \ref{p2.3}.
\end{proof}

\subsection{Birational sheaves and tensor products} In this subsection, we show that the tensor product in $\HI$ of two birational sheaves is birational. 

More correctly, let $\nu_0:\HI\to \HI^\o$ be the left adjoint of $i^\o$ (Proposition \ref{p5.5}): by Lemma \ref{lA.5},  $\nu_0 \sF = \sH_0(\nu_{\le 0}\sF[0])$. Then $\nu_0$ induces on $\HI^\o$ a tensor structure $\otimes_{\HI^\o}$ characterised by
\[\nu_0\sF\otimes_{\HI^\o} \nu_0 \sG = \nu_0(\sF\otimes_{\HI} \sG)\]
and this tensor product is right $t$-exact.

Let $\sF,\sG\in \HI^\o$. From the isomorphism
\[\nu_0(i^\o\sF\otimes_{\HI} i^\o \sG)\simeq \nu_0i^\o \sF\otimes_{\HI^\o} \nu_0i^\o \sG\iso \sF \otimes_{\HI^\o}\sG\]
we get by adjunction a morphism
\begin{equation}\label{eqHIo}
i^\o\sF\otimes_{\HI} i^\o \sG\to i^\o\left(\sF \otimes_{\HI^\o}\sG\right).
\end{equation}

\begin{thm}\label{t4.2} \eqref{eqHIo} is an isomorphism.
\end{thm}

The analogous result for the inclusion $\HI\inj \NST$ is well-known to be false: for example $\bG_m\otimes_{\NST}\bG_m$ is not homotopy invariant. Indeed, its quotient $\bG_m\otimes_{\HI} \bG_m$ verifies  $\bG_m\otimes_{\HI} \bG_m(F)=K_2^M(F)$ \cite[Th. 5.1]{mvw}, while $\bG_m\otimes_{\NST}\bG_m(F)=F^*\otimes F^*$ if $F$ is algebraically closed.

\begin{proof} For ease of notation, let us suppress the use of $i^\o$. So we must show that the sheaf $\sH=\sF\otimes_{\HI}\sG$ is birational. We shall use the characterisation of birational sheaves given in Proposition \ref{p2.2}. 

So, let $K/F$ be a function field, $C/K$ be a (proper) regular curve and $c\in C$ be a closed point. We must show that the map
\[\sH(\sO_{C,c})\to \sH(K(C))\]
is surjective.

By the surjectivity of  the map in \cite[(2.10)]{somekawa}, the composition
\[\bigoplus_{[E:K(C)]<\infty} \sF(E)\otimes \sG(E)\to \bigoplus_{[E:K(C)]<\infty} \sH(E)\by{(Tr_{E/K(C)})} \sH(K(C))\]
is surjective. 

For each $E$ as above, let $C_E$ be the proper regular model of $E/K$; let $f_E:C_E\to C$ be the canonical map and let $c_E=f_E^{-1}(c)$. Then $\sO_{C_E,c_E}$ is finite over $\sO_{C,c}$, hence we have a commutative diagram
\[\begin{CD}
\displaystyle\bigoplus_{[E:K(C)]<\infty} \sF(\sO_{C_E,c_E})\otimes \sG(\sO_{C_E,c_E})@>\sim>> \displaystyle\bigoplus_{[E:K(C)]<\infty}\sF(E)\otimes \sG(E)\\
@VVV @VVV\\
\sH(\sO_{C,c})@>>> \sH(K(C))
\end{CD}\]
where the top map is an isomorphism because $\sF$ and $\sG$ are birational. The theorem follows.
\end{proof}

\begin{rk} Let $E$ be an elliptic curve: one can show that $\Tor_1^{\DM}(E,E)$ is not birational, see \cite[Prop. 4.2]{rqnr}. Hence the functor $i^\o:\DM^\o\to \DM^\eff$ is not monoidal.
\end{rk}

\section{Birational sheaves and cycle modules} Let $\CM$ denote the category of cycle modules in the sense of Rost \cite{rost}. In \cite[Th. 6.2.3]{birat-pure} we defined  a pair of adjoint functors
\[K^?: \Mod\Chow^\o \leftrightarrows \CM : A^0 \]
where $A^0$ sends a cycle module $M=(M_n)_{n\in\Z}$ to a functor extending the assignment $X\mapsto A^0(X,M_0)$, and $K^?$ is fully faithful. The description of the essential image of $K^?$ was left open; we now have the material to answer this question. 

\subsection{Cycle modules and Somekawa $K$-groups} To formulate the answer, we need some preparation. Recall that, in \cite{Somekawa}, Somekawa associated an abelian group $K(F;G_1,\dots,G_n)$ to $n$ semi-abelian varieties $G_1,\dots,G_n$ over $F$; this definition was extended in \cite{somekawa} to objects $G_1,\dots,G_n$ of $\HI$.\footnote{Recall that a semi-abelian variety defines an object of $\HI$ by \cite[Lemma 3.2]{spsz} and \cite[Lemma 1.3.2]{bar-kahn}.} 

Let $M$ be a cycle module. Recall now that D\'eglise associated to $M$ a graded object $(\sM_n)$ of $\HI$ such that
\begin{equation}\label{eq4.1}
\sM_n(X)=A^0(X,M_n)
\end{equation}
for any $X\in \Sm$. We then have:

\begin{lemma}\label{l4.7} For any extension $E/F$, the map $ E^*\otimes M_n(E)\to M_{n+1}(E)$ from \cite[D3]{rost} induces a homomorphism
\[\theta_n:K(E;\bG_m,\sM_n)\to M_{n+1}(E).\]
\end{lemma}

\begin{proof} Recall from \cite[Def. 5.1]{somekawa} that $K(E;\bG_m,\sM_n)$ has generators $\{f,m\}_{E'/E}$ where $E'$ runs through the finite extensions of $E$, $f\in E^*$ and $m\in \sM_n(\Spec E)=M_n(E)$, subject to 3 types of relations directly generalising those of \cite[(1.2.0), (1.2.1) and (1.2.2)]{Somekawa}: bilinearity, the projection formula and relations ``of Somekawa type''.   Write $\phi:E\to E'$ for the inclusion. We define $\theta_n(\{m,f\}) = \phi^*(f\cdot m)$, where $\phi^*$ is the transfer map from \cite[D2]{rost} and the $f\cdot m$ is the product from \cite[D3]{rost}. We need to check that $\theta_n$ respects the relations. Bilinearity is obvious and the projection formula follows from \cite[R2b and R2c]{rost}. By \cite[Rk. 6.3]{somekawa}, to prove the relations of Somekawa type it now suffices to prove the finer relations ``of geometric type'' from \loccit, Def. 6.1; these follow directly from  \cite[(RC) in Prop. 2.2]{rost}.
\end{proof}

\begin{thm}\label{t4.1} Let $M\in \CM$. Then $M$ is in the essential image of $K^?$ if and only if:
\begin{thlist}
\item $M_n=0$ for $n<0$;
\item for any $n\ge 0$, the map $\theta_n$ of Lemma \ref{l4.7} is an isomorphism.
\end{thlist}
\end{thm}

\subsection{A reformulation} We first translate Theorem \ref{t4.1} into a statement involving homotopy invariant Nisnevich sheaves with transfers, whose proof is then rather straightforward. For this, we need to recall D\'eglise's theory from \cite{modhot} in further detail:

 By a construction going back to Morel and Voevodsky, the category $\DM^\eff$ can be fully embedded into a larger $\otimes$-triangulated category $\DM$ of ``$\Z(1)$-spectra''. There is an adjunction \cite[Prop. 4.7]{modhot}:
\[\Sigma^\infty: \DM^\eff \leftrightarrows \DM : \Omega^\infty \]
where $\Sigma^\infty$ is fully faithful, monoidal and $\Sigma^\infty\Z(1)$ is invertible. Moreover the homotopy $t$-structure of $\DM^\eff$ extends to a $t$-structure on $\DM$, with heart the category $\HI_*$ of \emph{homotopic modules} \cite[Th. 5.11]{modhot}; the functor $\Omega^\infty$ is $t$-exact. 

By  \cite[Def. 1.17]{modhot}, an object of $\HI_*$ is a sequence
$(\sM_n,\epsilon_n)_{n\in\Z}$, where $\sM_n\in\HI$ and $\epsilon_n$ is an isomorphism $\sM_n\iso \uHom_{\HI}(\bG_m,\sM_{n+1})$. By Proposition \ref{p4.6}, this may also be written $\sM_n\iso (\sM_{n+1})_{-1}$.\footnote{In \cite{modhot}, D\'eglise writes $S_t^1$ for $\bG_m$. These sheaves are isomorphic, as follows for example from \cite[Th. 3.4.2 (iii)]{voetri} applied to $C=\A^1-\{0\}$.} Morphisms are given componentwise. The adjoint functors 
\[\sigma^\infty: \HI \leftrightarrows \HI_* : \omega^\infty \]
induced by $\Omega^\infty$ and $\Sigma^\infty$ are given by
\begin{align*}
\omega^\infty \sM_* &= \sM_0\\
(\sigma^\infty \sF)_n &=
\begin{cases}
\sF\otimes_{\HI} \bG_m^{\otimes n} &\text{if $n\ge 0$}\\
\sF_n \text{ (Voevodsky's contraction)} &\text{if $n< 0$.}
\end{cases}
\end{align*}

D\'eglise's main result, \cite[Th. 3.7]{modhot}, provides an equivalence of categories between $\HI_*$ and  $\CM$. More precisely, his functor $\CM\to \HI_*$ sends a cycle module $M_*$ to a homotopic module $\sM_*$ given by \eqref{eq4.1}. His functor $\HI_*\to \CM$ sends a homotopic module $\sM_*$ to a cycle module $M_*$ such that
\[M_n(K) = \colim \sM_n(U)\]
for any function field $K/F$, where $U$ runs through the collection of smooth models of $K$ and transition maps are open immersions \cite[2.10 and 3.1]{modhot}.

On the other hand, the full embedding $D:\Cor_\rat^\o\to \BFC$ of Proposition \ref{p2.6} yields a pair of adjoint functors
\[D_!: \Mod\Chow^\o \leftrightarrows \HI^\o : D^*. \]

Recall finally that by \cite[(1.1)]{somekawa}, we have an isomorphism
\[K(E;\bG_m,\sM_n)\simeq (\bG_m\otimes_{\HI} \sM_n)(\Spec E)\]
where $\sM_n$ is associated to a given cycle module $M$ as in \eqref{eq4.1}.  
The translation we need for proving Theorem \ref{t4.1} is now provided by:

\begin{prop}\label{t5.2} The functors $K^?$ and $A^0$ are respectively isomorphic to the compositions $\sigma^\infty\circ i^\o\circ D_!$ and $D^*\circ R^0_\nr \circ \omega^\infty$.
\end{prop}

\begin{proof} By adjunction, it suffices to construct a natural isomorphism $K^?\simeq \sigma^\infty\circ i^\o\circ D_!$. By construction, $K^?$ is obtained as the left Kan extension of Merkurjev's functor
\[X\mapsto K^X\]
where, for a smooth projective variety $X$, $K^X$ is a cycle module such that $K_*^X(E)=A_0(X_E,K_*)$ for any function field $E/F$. Here $K$ denotes the cycle module given by Milnor $K$-theory. As $K^X$ represents the functor $M_*\mapsto A^0(X,M_0)$  \cite[Th. 2.10]{merk2}, its image in $\HI_*$ represents the functor $\sM_*\mapsto H^0_\Nis(X,\sM_0)$. But
\begin{multline*}
H^0(X,\sM_0)\simeq \Hom_{\HI}(h_0^\Nis(X),\sM_0)\\=\Hom_{\HI}(h_0^\Nis(X),\omega^\infty\sM_*)
\simeq \Hom_{\HI_*}(\sigma^\infty h_0^\Nis(X),\sM_*)
\end{multline*}
so that the image of $K^X$ in $\HI_*$ is $\sigma^\infty h_0^\Nis(X)$. Since $D_! y(h^\o(X)) = h_0^\Nis(X)$ by Corollary \ref{c1.1}, this concludes the proof.
\end{proof}

\subsection{Proof of Theorem \ref{t4.1}} In view of Proposition \ref{t5.2},  we have to show that the essential image of the composite functor $\HI^\o\by{i^\o}\HI\by{\sigma^\infty} \HI_*$
consists of those homotopic modules $\sF_*$ such that
\begin{thlist}
\item $\sF_n=0$ for $n<0$;
\item for any $n\ge 0$, the canonical map $\sF_0\otimes_{\HI} \bG_m^{\otimes n}\to \sF_n$ is an isomorphism.
\end{thlist} 

Any homotopic module in the essential image if $\sigma^\infty i^\o$ verifies (i) by Proposition \ref{p2.3} and (ii) by the above description of $\sigma^\infty$. For the converse, we use the fact that $i^\o$ has a right adjoint $R^0_\nr$ (Proposition \ref{p5.5}). Let $\sF_*\in \HI_*$ verify (i) and (ii): we must show that the counit map 
\[\sigma^\infty i^\o R^0_\nr \omega^\infty \sF_*\to \sF_*\]
is an isomorphism. We have $\omega^\infty\sF_*=\sF_0$; by Proposition \ref{p2.3}, $\sF_0\in\HI^\o$, hence $i^\o R^0_\nr\sF_0=\sF_0$ and the claim.
\qed

\section{Unramified cohomology}\label{s.unram}

\subsection{The functor $R_\nr$} In this section, we assume $F$ perfect. We study here the right adjoint to  $i^\o:\DM^\o\to\DM^\eff$ from Theorem \ref{c3.1}: this right adjoint is denoted by $R_\nr$.

 For any $C\in \DM^\eff$ and $q\in \Z$, we write $R^q_\nr C$ for
$\sH^q(R_\nr C)\in \HI^\o$, where $\sH^q$ corresponds to the homotopy
$t$-structure on $\DM^\o$.   By Theorem \ref{t3.2} and  Lemma \ref{lA.5}, $R_\nr$ is left exact with respect to the
homotopy $t$-structures of $\DM^\eff$ and $\DM^\o$. In particular,  $R^q_\nr \sF=0$ for $q<0$ if $\sF\in \HI$, and  Lemma \ref{lA.5} shows that the functor
\[\sF\mapsto R^0_\nr\sF\]
from $\HI$ to $\HI^\o$ is the right adjoint to the inclusion functor
$i^\o:\HI^\o\to \HI$ from Proposition \ref{p5.5}.

The main result of this section is Theorem \ref{t6.2}: for any $\sF\in\HI$, $R^0_\nr\sF$ coincides with the \emph{unramified part} of $\sF$ in the sense of classical unramified cohomology \cite{ct}. 

See \cite{rqnr} for computations of the ``higher derived functors of unramified cohomology" $R^q_\nr \sF$, for $q>0$ and certain $\sF$'s.

\subsection{The group $\sF_\nr(X)$} 

\begin{defn} Let $K/F$ be a finitely generated field extension.\\
a) A valuation  $v$  on $K$ (trivial on $F$) is \emph{divisorial} if  it is discrete of rank $1$ and its valuation ring is of the
form $O_{V,x}$, for $V$ a smooth $F$-scheme of finite type with function field $K$. (These are
the prime divisors of Zariski-Samuel, \cite[Ch. VI, \S 14]{zs}.)\\ 
b) For $\sF\in \HI$, we set
\[\sF_\nr(K/F)= \Ker\left(\sF(K)\by{(\partial_v)} \prod_v \sF_{-1}(F(v))\right)\]
where $v$ runs through all the divisorial valuations on $K$.\\
c) If $X$ is a smooth model of $K$, we set $\sF_\nr(X)=\sF_\nr(K/F)$. 
\end{defn}

 The exact sequence \eqref{eq6.1} implies:
 
 \begin{lemma}\label{l8.5} For any $X\in \Sm$, $\sF_\nr(X)\subseteq \sF(X)$.\qed
  \end{lemma}

We are going to show that $\sF_\nr$ is preserved by finite correspondences. We begin with a series of
lemmas.

\begin{lemma}\label{l6.6} Let $f:Y\to X$ be a dominant morphism of smooth irreducible varieties.
Then
$f^*\sF_\nr(X)\subset \sF_\nr(Y)$.
\end{lemma}

\begin{proof} Let $\alpha\in \sF_\nr(X)$, and let $w$ be a divisorial valuation on the
function field $L$ of
$Y$. Since $f$ is dominant, $L$ is an extension of $K$. Let $v$ be
the restriction of $w$ to $K$. Then 
\[\partial_w(f^*\alpha) = ef^*\partial_v(\alpha)\]
where $e$ is the ramification index if $v$ is nontrivial and $e=0$ if $v$ is trivial; we simply write $f:\Spec F(w)\to \Spec F(v)$ for the map
induced by $f$. This formula follows from examining the purity isomorphism of
\cite[Lemma 4.36]{voepre}. It also follows from  D\'eglise's theory of generic motives \cite[lemme 5.4.7]{motgen}, since the residues $\partial $ may be expressed in their terms \cite[3.1]{modhot}.
\end{proof}

\begin{lemma}\label{l6.7} Let $f:X\to Y$ be a finite surjective morphism of smooth
$F$-varieties. Let
${}^t f\in c(Y,X)$ be the transpose of its graph. Then $({}^t f)^*\sF_\nr(X)\subseteq
\sF_\nr(Y)$.
\end{lemma}

\begin{proof} Here, $K=F(X)$ is a finite extension of $L=F(Y)$. Let $\alpha\in \sF_\nr(X)$, and
let $v$ be a divisorial valuation of $L$. Then
\[\partial_v(({}^t f)^*\alpha) = \sum_{w\mid v} ({}^t f)^*\partial_w(\alpha)\]
\cite[lemme 5.4.7]{motgen}.
\end{proof}

We need the following version of a theorem of Knaf-Kuhlmann  \cite[Th. 1.1]{KK}.\footnote{Its proof obviously extends to compositions of more than $2$ divisorial valuations.}

\begin{prop}\label{l6.8} Let $K/F$ be a function field, $w$ a divisorial valuation of $K$ with residue field $L$, $v$ a divisorial valuation of $L$. Then there is a closed immersion $i:Z\inj W$  of smooth $F$-varieties such that  $K\simeq F(W)$, $L\simeq F(Z)$, $\sO_{W,Z}\simeq \sO_w$, and $v$ is finite on $Z$.
\end{prop}

\begin{proof} Let $u$ be the composite valuation, which is discrete of rank $2$ (here we identify valuations with the associated surjective places): it is an Abyankhar place in the sense of \cite{KK}. The residue field of $u$ is separably generated since $F$ is perfect. Let $\sO_u\subset K$ be the local ring of $u$ and $a,b\in \sO_u$ be two elements such that $u(a)=(1,0)$ and $u(b)=(0,1)$.  By  \cite[Th. 1.1]{KK}, there is a smooth model $W$ of $K$ on which $u$ has a centre $t$ of codimension $2$, and such that $a,b$ are $\sO_{W,t}$-monomials in $a_1,a_2$ (in the sense of \cite[p. 234]{KK}) for some regular system of parameters $(a_1,a_2)$ of $\sO_{W,t}$. This easily implies that $u(a_1)=(1,0)$ and $u(a_2)=(0,1)$, up to permuting $(a_1,a_2)$. Then $\sO_{W,t}/a_1\sO_{X,t}$ is regular \cite[Ch. VIII, \S 11, Th. 26]{zs}, and is the local ring of $t$ on $Z$, the closure of the centre of $w$ on $W$. Thus $Z$ is regular at $t$, hence smooth around $t$ since $F$ is perfect, and we can make it smooth by shrinking $W$.
\end{proof}

\begin{lemma}[Main lemma]\label{l6.9} Let $i:Y\to X$ be a closed immersion of codimension $1$.
Then
$i^*\sF_\nr(X)\subseteq \sF_\nr(Y)$.
\end{lemma}

\begin{proof} Let $v$ be a divisorial valuation of $F(Y)$, and let  $w$ be the divisorial valuation on $K=F(X)$ defined by $i:Y\to X$.  Let $u$ be the composite valuation, and let $(W,Z)$ be as in the conclusion of Proposition \ref{l6.8}. We thus have two closed immersions of codimension $1$:
\begin{align*}
i:&Y\to X\\
i':&Z\to W.
\end{align*}

Since $w$ is finite on $X$ and $W$, they share an open neighbourhood $U$ of $\Spec O_w$. Let $T$ be the closure of the centre of $w$ in $U$. Since $Y,Z$ and $T$ are birational, they share a nonempty open subset $U'$. We now have the following situation:
\[\begin{CD}
Y@>i>> X\\
@AAA @AAA\\
U'@>i''>> U\\
@VVV @VVV \\
Z@>i'>> W
\end{CD}\]
where all varieties are smooth, vertical maps are open immersions, $i$ and $i'$ are closed immersions while $i''$ is locally closed. This in turn gives a commutative diagram
\[\begin{CD}
\sF(Y)@<i^*<< \sF_\nr(X)\\
\bigcap && || \\
\sF(U')@<{i''}^*<< \sF_\nr(U)\\
\bigcup&& ||\\
\sF(Z)@<{i'}^*<< \sF_\nr(W).
\end{CD}\]

Here we identified the three left-hand terms with subgroups of $\sF(L)$ thanks to \eqref{eq6.1}, where $L=F(Y)$. Let $\alpha\in \sF_\nr(X)$. The diagram shows that $i^*\alpha\in \sF(Y)$ lies in $\sF(Z)$. Since $v$ is finite on $Z$, we have $\partial_v(i^*\alpha)=0$ by \eqref{eq6.1} applied to $Z$. Since $v$ was an arbitrary divisorial valuation of $L$, this shows that $i^*\alpha$ is unramified, as requested.
\end{proof}

\subsection{$\sF_\nr$ and $R^0_\nr\sF$} The following theorem justifies the notation $R_\nr\sF$:

\begin{thm}\label{t6.2} Let $\sF\in \HI$. Let $X\in \Sm$ be  irreducible with function field $K$. Then
there is a natural isomorphism
\[R^0_\nr\sF(X) = \sF_\nr(X). \]
\end{thm}

\begin{proof} Recall that, by Proposition \ref{p5.5}, the counit map
\[i^\o R_\nr^0\sF\to \sF\]
is a monomorphism. By Lemma \ref{l8.5}, we may thus identify both groups with subgroups of $\sF(X)$. We proceed in two steps:

\begin{enumerate}
\item $R^0_\nr\sF(X)\subset \sF_\nr(X)$.
\item $X\mapsto \sF_\nr(X)$ defines an object of $\HI^\o$.
\end{enumerate}

Granting (1) and (2), the theorem follows from the universal property of $R^0_\nr\sF$.

(1) Let $\alpha\in R^0_\nr\sF(X)$, and let $v$ be a divisorial valuation of $K$. We want to
show that $\partial_v(\alpha)=0$. Choose a smooth model $V$ of $K$ on which $v$ is finite, with centre
a point $x$ of codimension $1$. Let $U$ be a common open subset of $X$ and $V$. Since
$R^0_\nr\sF\in \HI^\o$, we have isomorphisms
\[R^0_\nr\sF(X)\osi R^0_\nr\sF(U)\iso R^0_\nr\sF(V)\subset \sF(V)\]
and the claim follows from the complex \eqref{eq6.1}.

(2) By Proposition \ref{l14.1} a), it suffices to show that $\sF_\nr$ defines a sub-presheaf with transfers of $\sF$. Let  $\phi\in
c(Y,X)$ be a finite correspondence, with $Y$ smooth irreducible. We have to show that
$\phi^*\sF_\nr(X)\subseteq \sF_\nr(Y)$. For this, we may assume that $\phi$ is
defined by an irreducible subset $Z\subset Y\times X$.

Let $p:Z\to Y$ be the projection. There is a nonempty open subset $U\subseteq Y$ such that
$p^{-1}(U)$ is smooth. The transpose of the graph of $p_{|p^{-1}(U)}$ defines a finite
correspondence ${}^t p\in c(U,p^{-1}(U))$. Let $k$ be the immersion
$p^{-1}(U)\to Z\to Y\times X$ and $\gamma_k:p^{-1}(U)\to p^{-1}(U)\times Y\times X$ be the
associated graph map. Then the diagram
\[\xymatrix{
p^{-1}(U)\ar[r]^{\gamma_k} &p^{-1}(U)\times Y\times X\ar[dd]^\pi\\
U\ar[u]^{{}^t p}\ar[d]\\
Y\ar[r]^\phi & X
}\]
commutes in $\SmCor$. Note that $\gamma_k$ is a regular embedding, hence may be locally written as a composition of closed embeddings of smooth varieties, of relative dimension $1$. Lemmas \ref{l6.6}, \ref{l6.9} and \ref{l6.7} then respectively show that $\pi^*$, $\gamma_k^*$ and $({}^tp)^*$ respect unramified elements, and thus so does $\phi^*$.
\end{proof}

\begin{cor} Let $\sF\in \HI$. Suppose that a smooth variety $X$ has a smooth
compactification $\bar X$. Then 
$R^0_\nr\sF(X)$ is given by the formula
\[R^0_\nr\sF(X)=\sF(\bar X).\]
\end{cor}

\begin{proof} By  Theorem \ref{t6.2}, we may replace $R^0_\nr\sF(X)$ by $\sF_\nr(X)$; the conclusion is then classical \cite[Prop. 2.1.8 e)]{ct}. Note that the ``codimension $1$ purity'' hypothesis is satisfied in view of Theorem \ref{t2.1}. \end{proof}

\appendix

\section{Localisation, Brown representability and $t$-structures}

In this appendix, we collect technical results on triangulated categories which are used in the main body of the paper.   \S \ref{r6.1} recollects results of Neeman on compactly generated triangulated categories \cite{neeman,neeman2}, also revisited by Beilinson-Vologodsky \cite{be-vo}; the existence of the functor $R_\nr$ rests on Theorem \ref{neem3} (v). \S \ref{s.loc} recalls the behaviour of $t$-structures under adjoints and localisations. The main result is Theorem \ref{tC.1}, from which Theorem \ref{t3.0} is deduced almost directly.

\subsection{Terminology} \label{s.term} It is worthwhile to first recall and fix some terminology
on triangulated categories: we follow Neeman in \cite{neeman2}. As has become widespread, we replace the old terminology ``exact functor'' by ``triangulated functor'', and ``distinguished triangle'' by ``exact triangle''.

Let $\sT$ be a triangulated category. Recall that a strictly full subcategory $\sS$ of $\sT$ is \emph{triangulated} if it is additive and closed under the formation of shifts and cones. Then one can define the \emph{Verdier quotient} $\sT/\sS$ \cite{verd}: this is a triangulated category which comes with a triangulated functor $\sT\to \sT/\sS$, universal among triangulated functors sending all objects of $\sS$ to $0$ \cite[Cor. II.2.2.11 c)]{verd}. We say that a triangulated functor $T:\sT\to \sU$ is a \emph{localisation} if it induces an equivalence of categories $\sT/T^{-1}(0)\iso \sU$. A triangulated subcategory $\sS$ of $\sT$
is \emph{thick} (\resp \emph{localising}) if it is stable under representable direct
summands (\resp and under representable direct sums). If $\sS^\natural\subseteq \sT$ is the smallest thick subcategory of $\sT$ containing $\sS$, the functor $\sT/\sS\to \sT/\sS^\natural$ is an equivalence of categories \cite[Cor. II.2.2.11 a)]{verd}. We have:

\begin{lemma}[\protect{\cite[Cor. 3.2.11]{neeman2}}]\label{lA.2} Let $\sT$ be a triangulated category with small direct sums. Let $\sS$ be a localising subcategory. Then $\sT/\sS$ has small direct sums, and the universal functor $\sT \to \sT/\sS$ preserves coproducts.
\end{lemma}

Recall that an object $X$ of  $\sT$ is \emph{compact} if  
$\sT(X,-)$ commutes with representable direct sums.  A 
triangulated functor $T:\sT\to\sT'$ 
is \emph{dense} if the image of $T$ generates $\sT'$ in the sense that
$\sT'$ is the smallest localising subcategory of itself that contains
this image: this is the same notion as in \cite{voetri}. 

\begin{defn}\label{dA.2}
Let $B$ be a class of objects in $\sT$. We write
\begin{align*}
B^\perp &=\{X\in \sT\mid \sT(B[i],X)=0\,\forall i\in\Z\}\\
{}^\perp B &=\{X\in \sT\mid \sT(X,B[i])=0\,\forall i\in\Z\}.
\end{align*}
These are triangulated subcategories of $\sT$. We also write
\begin{itemize}
\item $\langle B\rangle$ for the triangulated subcategory of $\sT$ generated by $B$ (\ie the smallest triangulated subcategory of $\sT$ which contains $B$);
\item  $\langle B\rangle^\natural$ for the thick subcategory of $\sT$ generated by $B$;
\item  $\langle B\rangle^\oplus$ for the localising subcategory of $\sT$ generated by $B$.
\end{itemize}
\end{defn}

\subsection{Some results of Neeman}\label{r6.1}

Below we shall use results from Amnon Neeman's book
\cite{neeman2}, especially \emph{loc. cit.}, Lemma 4.4.5 and Theorem 4.4.9. Most of them are stated there with respect to an infinite cardinal $\alpha$; here, we shall only need the case where $\alpha=\aleph_0$ as in \cite{neeman}. For the reader's convenience, we now state these results in this special case.

\begin{thm}[\protect{\cite[Th. 4.3.3 and Cor. 4.4.5]{neeman2}}]\label{neem1} Let $\sT$ be a triangulated category with small direct sums, $\sT^c$ its (thick) subcategory of compact objects, $S$ a thick subcategory of $\sT^c$ and $\langle S \rangle^\oplus$ the localising
subcategory of $\sT$ generated by $S$. Let $(x,z)\in \sT^c\times \langle S \rangle^\oplus$, and let $f\in \sT(x,z)$. Then $f$ factors through an object of $S$. In
particular, $\langle S \rangle^\oplus =\sT\Rightarrow S=\sT^c$.
\end{thm}

Let us sketch the proof: ``in particular" is obtained by applying the
theorem to the identity map of a compact object. The proof of \ref{neem1}
goes as follows: Neeman introduces the full subcategory $\sS$ of $\sT$
consisting of those objects $z$ for which the conclusion of the theorem is
valid for any compact $x$. He successively proves that $\sS$ contains $S$, is
triangulated and is closed under coproducts. Therefore $\sS$ contains $\langle S \rangle^\oplus$.

This theorem is used in the proof of Proposition \ref{l8.1} below.

\begin{thm}[\protect{\cite[Th. 4.4.9]{neeman2}}]\label{neem2} Let $\sS$ be a compactly generated  triangulated category with small direct sums, and let $R$ be a set of objects of $\sS^c$. Write $\sR$ for the localising subcategory generated by $R$. Then:
\begin{thlist}
\item $\sR^c=\sR\cap \sS^c$. In particular, $\sR=\langle \sR^c\rangle^\oplus$.
\item The natural functor $\sS^c/\sR^c\to \sS/\sR$ factors through a full embedding $\sS^c/\sR^c\to (\sS/\sR)^c$.
\item Any object of $(\sS/\sR)^c$ is isomorphic to a direct summand of an object of $\sS^c/\sR^c$.
\end{thlist}
\end{thm}

The next result we shall use is Neeman's ``Brown representability theorem", which gives sufficient conditions for the existence of a right adjoint.

\begin{defn}[\protect{\cite[Def. 8.2.1]{neeman2}}] Let $\sT$ be a triangulated category. We say that $\sT$ has the \emph{Brown representability property} if 
\begin{enumerate}
\item It has small direct sums.
\item Any homological functor $H:\sT^\op\to \Ab$ which converts infinite direct sums into products is representable.
\end{enumerate}
\end{defn}

\begin{lemma}\label{l3.2} a) Any adjoint (left or right) of a triangulated functor is triangulated.\\ b) Suppose that $\sT$ has the Brown representability property. Let $f:\sT\to \sU$ be a triangulated functor.  Then $f$ has a right adjoint if and only if it commutes with infinite direct sums. \end{lemma}

\begin{proof} a) is proven in \cite[Lemma 5.3.6]{neeman2}. We give the proof of b) since it is very simple. If $f$ has a right adjoint, it commutes with all representable colimits. Conversely, let $U\in \sU$. We must prove that the functor $T\mapsto \sU(fT,U)$ is representable. But if $f$ commutes with infinite direct sums, this functor converts infinite direct sums into products.
\end{proof}

\begin{thm}[\protect{\cite[Prop. 8.4.2]{neeman2}}]\label{neem5}  If $\sT$ has small direct sums and is compactly generated, it has the Brown representability property.
\end{thm}

We now get the following complement to Theorem \ref{neem2}:

\begin{thm} \label{neem3} With the assumptions and notation of Theorem \ref{neem2}, 
\begin{thlist}
\item The localisation functor 
$\sS\by{\pi} \sS/\sR$ has a right adjoint $j$.
\item The essential image of $j$ is 
$\sR^\perp$.
\item Let $i:\sR\to \sS$ be the inclusion functor. Then $i$ has a right adjoint $p$, and for any object $X\in \sS$, the sequence
\[ipX\to X\to j\pi X\]
defines an exact triangle.
\item $\langle(\sS/\sR)^c\rangle^\oplus=\sS/\sR$. 
\item The functor $j$ itself has a right adjoint.
\end{thlist}
\end{thm}

\begin{proof} By  assumption, $\sR$ has small direct sums. It is compactly generated by Theorem \ref{neem2}. Hence it has the Brown representability property by Theorem \ref{neem5}. Lemma \ref{l3.2} now implies that the functor $i$ of (iii) has a right adjoint.

Given this, assertions (i), (ii) and (iii) are part of a general theorem of Verdier \cite[Prop. 2.3.3]{verd}.

For (iv), we have $\sS/\sR = (\sS^c/\sR^c)^\oplus$ since $\sS=(\sS^c)^\oplus$ and $\pi$ commutes with small direct sums (Lemma \ref{lA.2}), and we conclude by Theorem \ref{neem2} (iii). 

For (v), let $j':\sR^\perp \to \sS$ be the inclusion. Observe  that $ \sR^\perp= (\sR^c)^\perp$ by denseness; by Theorem \ref{neem2} (i) this implies that $j'$ commutes with small direct sums, hence, by (ii), so does $j$. Since,  by (iv), $\sS/\sR$ is compactly generated, it has the Brown representability property which guarantees that $j$ has a right adjoint by Lemma \ref{l3.2} b).
\end{proof}

\begin{rk}\label{t3.1} The proposition p. 1714 of  Bei\-lin\-son-Vologodsky \cite{be-vo} wraps up all the above, except for the existence of the right adjoint to $j$ in Theorem \ref{neem3} (v). It adds a nice explicit description of the objects of $\sR$: every such object can be represented as $\hocolim(M_a, i_a)$
where $M_0$ and each $cone(i_a)$ is a direct sum of translations of objects from $R$. We shall not use this result here.
\end{rk}

\subsection{Localisation and $t$-structures} \label{s.loc}
The standard reference for $t$-structures is \cite{bbd}, whose notations we follow. We shall mainly use the following lemma \cite[Prop. 1.3.17 (i), (iii)]{bbd}:

\begin{lemma}\label{lA.5} Let $F^*:\sS\leftrightarrows \sT :F_*$ be a pair of adjoint triangulated functors between $t$-categories $\sS,\sT$ with hearts $\sA,\sB$. Then $F^*$ is right $t$-exact (\ie $F^*(\sS^{\le 0})\subseteq \sT^{\le 0}$) if and only if $F_*$ is left exact (\ie $F_*(\sT^{\ge 0})\subseteq \sS^{\ge 0}$). In this case, the functor ${}^pF^*:\sA\ni A\mapsto H^0F^*(A)\in \sB$ is  right exact, ${}^pF_*:\sB\ni B\mapsto H^0F_*(B)\in \sA$ is left exact and ${}^pF^*:\sA\leftrightarrows \sB :{}^pF_*$ form a pair of adjoint functors.\qed
\end{lemma}

We shall also need this lemma in the proof of Theorem \ref{tC.1}:

\begin{lemma}\label{lC.1} Let $F:\sS\to \sT$ be a right exact $t$-functor between $t$-categories $\sS,\sT$ with hearts $\sA,\sB$. Assume that ${}^pF:\sA\to \sB$ has kernel $0$  and that the $t$-structure of $\sS$ is non-degenerate.  Then $F$ is conservative in the following two cases:
\begin{thlist}
\item $F$ is $t$-exact;
\item the $t$-structure of $\sS$ is bounded above.
\end{thlist}
\end{lemma}

\begin{proof} Let $X\in \sS$ be such that $F(X)=0$: we must show that $X=0$. In case (i), we just use the isomorphism ${}^pH^i(F(X)) ={}^pF ({}^pH^i(X))$ for any $i\in\Z$. In case (ii), let $i$ be an integer such that ${}^pH^j(X)= 0$ for $j>i$. By right exactness,
\[0= {}^pH^i(F(X)) ={}^pF ({}^pH^i(X))\]
hence ${}^pH^i(X)=0$ and we conclude.
\end{proof}

\subsection{The homotopy category of an additive category}\label{A.C}

Throughout this section, $\sA$ is an essentially small additive category and $\Mod\sA$ is the category of right $\sA$-modules (see \S \ref{s.not}).

The following derived analogue of Proposition \ref{p5.3} is a special case of a theorem of Bernhard Keller (see \cite[Remark 5.3 (a)]{keller}). 

\begin{prop}\label{l8.1} The functor
\[K^b(\sA)\by{\iota_\sA} D(\Mod\sA)\]
induced by the Yoneda embedding is fully faithful, has dense image and
identifies
$K^b(\sA)^\natural$ with the full subcategory of compact objects of
$D(\Mod\sA)$.
\end{prop}

(For a self-contained proof, see the first version of this paper at \url{https://arxiv.org/abs/1506.08385}.)

 Note that $D(\Mod\sA)$ is pseudo-abelian since it has representable infinite direct sums, which justifies the last assertion. Also, by  Balmer-Schlichting \cite{bs}, the categories $K^b(\sA)^\natural$ and $K^b(\sA^\natural)$ are equivalent, although we shall not use it.

We now go back to the results of \S \ref{r6.1}. By Proposition \ref{l8.1}, the category $\sS=D(\Mod\sA)$ verifies the hypotheses of Theorem \ref{neem2}, with $\sS^c=K^b(\sA)^\natural$. Thus, if $R$ is a set of objects of $K^b(\sA)$, $\sR=\langle R\rangle\subseteq K^b(\sA)$ and $\sR^\oplus=\langle y(R)\rangle^\oplus\subseteq D(\Mod\sA)$, the conclusions of Theorems \ref{neem2} and \ref{neem3} apply. So:
\begin{enumerate}
\item $\sR^c=\sR^\natural$ and $(D(\Mod \sA)/\sR^\oplus)^c=(K^b(\sA)/\sR)^\natural$; in particular, $K^b(\sA)/\sR\to D(\Mod \sA)/\sR^\oplus$ is fully faithful and dense.
\item The projection functor $D(\Mod\sA)\to D(\Mod\sA)/\sR^\oplus$ has a right adjoint, which itself has a right adjoint.
\end{enumerate}

Let now $S$ be a set of morphisms in $\sA$ which contains all identities and is stable under direct sums. By \cite[Th. A.3.3]{birat-pure}, the category $\sB=S^{-1}\sA$ is additive, as well as the localisation functor $Q:\sA\to \sB$. Thus the setting of \S \ref{s.not} applies, and the functor $Q^*:\Mod\sB\to \Mod\sA$ is fully faithful. We may identify $S$ to a set of morphisms in $K^b(\sA)$ via the natural embedding $\eta:\sA\ni A\mapsto A[0]\in K^b(\sA)$. If we take for $R=R_S$ the set of cones of $\eta(s)$ for $s\in S$, it is natural to ask about the relationship between the above localisations and the categories $K^b(\sB)$, $D(\Mod\sB)$. The answer is given by the following theorem. 

\begin{thm}\label{tC.1} Let $R_S$ be as above, and write $\sR_S, \sR_S^\oplus$ for the corresponding triangulated subcategories of $K^b(\sA)$ and $D(\Mod\sA)$.\\
 a) The functor $K^b(Q):K^b(\sA)\to K^b(\sB)$ factors through $K^b(\sA)/\sR_S$. The functor $LQ_!:D(\Mod\sA)\to D(\Mod\sB)$ factors through $D(\Mod\sA)/\sR_S^\oplus$. This yields a naturally commutative diagram
\[\xymatrix{
\sA\ar[d]_Q\ar[r]^{\eta_\sA} &K^b(\sA)\ar[d]_{\bar Q}\ar[r]^{\iota_\sA}& D(\Mod\sA)\ar[d]_{\bar Q_!}\\
\sB\ar[r]^{\bar \eta_\sA}\ar[rd]^{\eta_\sB}&K^b(\sA)/\sR_S\ar[r]^{\bar\iota_\sA}\ar[d]^w& D(\Mod\sA)/\sR_S^\oplus\ar[d]^{w_\oplus}\\
&K^b(\sB)\ar[r]^{\iota_\sB}&D(\Mod\sB)
}\]
in which all functors not starting from $\sA$ or $\sB$ are triangulated.\\
b) The functor $\bar\iota_\sA$ is fully faithful, has dense image, and identifies $(K^b(\sA)/\sR_S)^\natural$ with the full subcategory of compact objects of $D(\Mod\sA)/\sR_S^\oplus$.\\
c) The functor $\bar Q_!$ has a (fully faithful) right adjoint $\bar Q^*$, which itself has a right adjoint $\bar Q_*$. The essential image of $\bar Q^*$ is $D_{\Mod\sB}(\Mod\sA)$ (see \eqref{eq3.4}), where $\Mod\sB$ is embedded in $\Mod\sA$ by means of $Q^*$.\\
d) Via $\bar Q^*$, the natural $t$-structure of $D(\Mod\sA)$ induces a $t$-structure on $D(\Mod\sA)/\sR_S^\oplus$, with heart $\Mod\sB$; the functor $\bar Q_!$ (\resp $\bar Q_*$) is  right (\resp left) $t$-exact.\\
e) The functor $\bar \eta_\sA$ is fully faithful; for $B_1,B_2\in \sB$ and $M\in \Mod\sB$, we have 
\begin{align*}
(K^b(\sA)/\sR_S)(\bar\eta_\sA(B_1),M[i])&=0 \quad \text{for } i\ne 0\\
(K^b(\sA)/\sR_S)(\bar\eta_\sA(B_1),\bar\eta_\sA(B_2)[i])&=0 \quad \text{for } i>0.
\end{align*}
f) The obvious functor
\[\phi:D(\Mod\sB)\to D_{\Mod\sB}(\Mod\sA)\simeq D(\Mod\sA)/\sR_S^\oplus\]
is right adjoint to $w_\oplus$, $t$-exact and conservative; it induces the identity on the hearts.\\
g) The functor $w_\oplus$ is right $t$-exact and induces the identity on the hearts; its restriction to $D^-(\Mod\sA)/\sR_S^\oplus$ is conservative. The functor $w$ is conservative as well.\\
h) If $Q_!:\Mod\sA\to \Mod\sB$ is exact, $w_\oplus$ is an equivalence of categories, and so is $w$ after pseudo-abelian completions.
\end{thm}

In part a) if this theorem, note that the total left derived functor $LQ_!$ exists, \eg by \cite[Th. 14.4.3]{ka-sch}.

\begin{proof} a) is obvious since $S$ gets inverted in $K^b(\sB)$ and $LQ_!y(A)[0]=y(Q(A))[0]$ for $A\in \sA$. b) and c) only repeat the points (1) and (2) above, except for the description of the image of $\bar Q^*$. 

Let $C\in D(\Mod\sA)$: by definition, $C\in \IM\bar Q^*$ if and only if the map $D(\Mod\sA)(y(B)[i],C)\by{s^*} D(\Mod\sA)(y(A)[i],C)$ is an isomorphism for any $s\in S$, $s:A\to B$, and any $i\in\Z$. Since $y(A)$ and $y(B)$ are projective in $\Mod\sA$, this isomorphism may be rewritten: $H_i(C)(B)\iso H_i(C)(A)$. Thus $C\in \IM\bar Q^*$ $\iff$ $H_i(C)\in \IM Q^*$ for all $i\in \Z$. d) follows from c) via Lemma \ref{lA.5}.

In e), since $Q:\sA\to \sB$ is essentially surjective we may write $B_i\simeq Q(A_i)$ for $A_1,A_2\in\sA$. Using the full faithfulness of $\bar \iota_\sA$, the first vanishing then follows from adjunction and the projectivity of $y(A_1)$, and the second one follows from the first and the right $t$-exactness of $\bar Q_!$. It remains to prove the full faithfulness of $\bar\eta_\sA$: we have
\begin{multline*}
(K^b(\sA)/\sR_S)(\bar\eta_\sA(Q(A_1)),\bar\eta_\sA(Q(A_2)))\iso\\
 (D(\Mod\sA)/\sR_S^\oplus)(\bar Q_!y(A_1)[0],\bar Q_!y(A_2)[0])\\
\simeq  D(\Mod\sA)(y(A_1)[0],\bar Q^*\bar Q_!y(A_2)[0])\\
\simeq  H^0(\bar Q^*\bar Q_!y(A_2))(A_1)=Q^*Q_!y(A_2)(A_1).
\end{multline*}

Here we used again the right $t$-exactness of $\bar Q_!$, plus Lemma \ref{lA.5}. But we have
\begin{multline*}
Q^*Q_!y(A_2)(A_1)= Q_!y(A_2)(Q(A_1))\\=y(Q(A_2)(Q(A_1))=\sB(Q(A_1),Q(A_2))
\end{multline*}
which concludes the proof.

f) follows from the adjunction $(\bar Q_!,\bar Q^*)$, the adjunction $(LQ_!,RQ^*)$ and the dual of Proposition \ref{padj} a); conservativity follows from Lemma \ref{lC.1} (i).

In g), the first two assertions follow from f) in view of Lemma \ref{lA.5}. The next claim is a special case of Lemma \ref{lC.1} (ii). The conservativity of $w$ now follows from the full faithfulness of $\iota_\sB$.

Let us prove h). In view of f), to show that $w_\oplus$ is an equivalence of categories it suffices to show that so is $\phi$.  Since $Q_!$ and $Q^*$ are exact,  the identity
\[L(Q_!Q^*) \simeq LQ_! LQ^*\]
holds trivially in $D(\Mod\sB)$. In particular, the counit map $LQ_!LQ^*\Rightarrow Id$ is an isomorphism and $LQ^*$ is fully faithful. Since $\bar Q^*$ is also fully faithful, we find that $\phi$ is fully faithful. Since it clearly commutes with infinite direct sums, its essential image $\sI$ is localising and to prove the essential surjectivity of $\phi$ it remains to show that $\sI$ is dense. Using b), we reduce to prove that $\sI$ contains the image of $\bar\iota_\sA\bar\eta_\sA$.

Let $B\in \sB$ and $A\in \sA$ such that $B=Q(A)$. Then $\bar\iota_\sA\bar\eta_\sA(B)=\bar Q_! y(A)[0]$, while $\phi\iota_\sB\eta_\sB(B)=LQ^*y(B)[0]=LQ^*LQ_!y(A)[0]$. We must show that the cone of the counit map $LQ^*LQ_!y(A)[0]\to y(A)[0]$ belongs to $\sR_S^\oplus$, or equivalently that it is left orthogonal to $\sI$. It suffices to show that it is left orthogonal to $LQ^*M[i]$ for any $M\in \Mod\sB$ and any $i\in \Z$. This follows easily from the full faithfulness of $LQ^*$.

The claim for $w$ now follows, since the equivalence $w_\oplus$ induces an equivalence between the subcategories of compact objects (see b) and Proposition \ref{l8.1}).
\end{proof}

\begin{rks} 1) One can show that,  in h), $Q_!$ is exact provided $S$ admits a calculus of right fractions (compare \cite[I.3, Prop. 1.1]{gz}).\\
2) In the terminology of Bondarko \cite[Def. 4.3.1 1]{bondarko}, Theorem \ref{tC.1} e) says that the full subcategory $\sB\by{\bar \eta_\sA} K^b(\sA)/\sR^\o_S$ is \emph{negative}. Since $\bar\eta_\sA(\sB)$ generates $K^b(\sA)/\sR^\o_S$, the latter carries a \emph{weight structure} with heart $\bar\eta_\sA(\sB)^\natural$ by \loccit, Th. 4.3.2. One can then check that the functor $w$ coincides with the weight complex functor $t$ of \loccit, Th. 3.3.1, which is conservative by part V of the latter theorem. Theorem \ref{tC.1} provides an alternative proof of this conservativity, using the natural $t$-structure of $D(\Mod\sA)/\sR_S$: this seems related to Bondarko's notion of adjacence between a weight structure and a $t$-structure \cite[\S 4.4]{bondarko}. In \cite[Th. 4.2.2]{bond-sos}, Bondarko and Sosnilo give a direct proof of Theorem \ref{tC.1} e), without using the full embedding $K^b(\sA)/\sR_S\inj D(\Mod\sA)/\sR_S^\oplus$.
\end{rks}

\begin{ex} Let $\sX$ be an additive subcategory of $\sA$ and let $\sI_\sX$ be the ideal of morphisms in $\sA$ which factor through an object of $\sX$, compare \cite[Ex. 1.3.1]{ak}: the projection functor $Q:\sA\to\sA/\sI_\sX$ is universal among additive functors mapping all objects of $\sX$ to $0$. Then $Q$ is a localisation. Indeed, let
\[S_\sX=\{s\in Ar(\sA)\mid s \text{ becomes invertible in } \sA/\sI_\sX\}.\]

Since $S_\sX$ contains all identities and is stable under direct sums, the localisation $S_\sX^{-1}\sA$ is additive \cite[Th. A.3.3]{birat-pure}; to show that the natural functor $S_\sX^{-1}\sA\to \sA/\sI_\sX$ is an equivalence of categories, it suffices to show that any object $X\in\sX$ maps to $0$ in $S_\sX^{-1}\sA$. Let $s:0\to X$ and $t:X\to 0$ be the canonical maps. Then $st$ and $ts$ both become invertible in $\sA/\sI_X$, hence $s\in S_\sX$.
\end{ex}


\begin{thebibliography}{SGA4-I]}
\bibitem{ak} Y. Andr\'e, B. Kahn {\it Nilpotence, radicaux et structures 
mono\"{\i}dales} (with an appendix of P. O'Sullivan), Rend. Sem.
Mat. Univ. Padova {\bf 108} (2002), 107--291.
\bibitem{bs} P. Balmer, M. Schlichting {\it Idempotent completion of
triangulated categories}, J. Algebra {\bf 236} (2001), 819--834.
\bibitem{bar-kahn} L. Barbieri-Viale, B. Kahn {\it On the derived category of $1$-motives}, Ast\'erisque {\bf 381}, 2016. 
\bibitem{bei} A. Beilinson {\it Remarks on $n$-motives and correspondences
at the generic point}, {\it in} Motives, Polylogarithms and Hodge Theory,
F. Bogomolov and L. Katzarkov, eds., International Press, 2002, 33--44.
\bibitem{bbd} A. Beilinson, J. Bernstein, P. Deligne {\it Faisceaux
pervers}, Ast\'erisque {\bf 100}, 1984.
\bibitem{be-vo} A. Beilinson, V. Vologodsky {\it  A DG guide to Voevodsky's motives}, GAFA (Geom. Funct. Anal.) {\bf 17} (2008), 1709--1787. 
\bibitem{bn} M. B\"okstedt, A. Neeman {\it Homotopy limits in
  triangulated categories}, Compositio Math. {\bf 86} (1993),
  209--234. 
\bibitem{bondarko} M. Bondarko {\it Weight structures vs. $t$-structures; weight filtrations, spectral sequences, and complexes (for motives and in general)}, J. $K$-theory {\bf 6} (2010), 387--504.
\bibitem{bond-sos} M. Bondarko, V. Sosnilo {\it Non-commutative localizations of additive categories and weight structures; applications to birational motives}, \url{http://arxiv.org/abs/1304.6059}.
\bibitem{ct} J.-L. Colliot-Th\'el\`ene {\it Unramified cohomology,
  birational invariants and the Gersten conjecture}, Proc. Symp. pure Math. {\bf 58} (I), AMS, 1995, 1--64.
\bibitem{motgen} F. D\'eglise {\it Motifs g\'en\'eriques}, Rend. Sem. Mat. Univ. Padova {\bf 119} (2008), 173--244.
\bibitem{modhot} F. D\'eglise {\it Modules homotopiques}, Doc. Math. {\bf 16} (2011), 411--455.
\bibitem{frivoe} E. Friedlander, V. Voevodsky {\it Bivariant cycle
cohomology}, {\it in} E. Friedlander, A. Suslin, V. Voevodsky Cycles,
transfers and motivic cohomology theories, Ann. Math. Studies {\bf 143},
Princeton University Press, 2000, 138--187. 
\bibitem{fulton} W. Fulton Intersection theory, Springer, 1984. 
\bibitem{gabber} L. Illusie, Y. Laszlo and F. Orgogozo (avec la collaboration de F. D\'eglise, A. Moreau, V. Pilloni, M. Raynaud, J. Riou, B. Stroh, M. Temkin et W. Zheng), Travaux de Gabber sur l'uniformisation locale et la cohomologie \'etale des sch\'emas quasi-excellents. S\'eminaire \`a l'\'Ecole polytechnique 2006-2008,   Ast\'erisque {\bf 363--364} (2014), xxiv+619 pages.
\bibitem{gz} P. Gabriel, M. Zisman Calculus of fractions and homotopy
theory, Springer, 1967.
\bibitem{godement} R. Godement Topologie alg\'ebrique et th\'eorie des faisceaux, Hermann, 1959.
\bibitem{hartshorne} R. Hartshorne Residues and Duality, Lect. Notes in
Math. {\bf 20}, Springer, 1966.
\bibitem{huber-ayoub} A. Huber {\it Slice filtration on motives and the Hodge
conjecture} (with an appendix by J. Ayoub), Math. Nachr. {\bf 281} (2008), 1764--1776.
\bibitem{motiftate} A. Huber, B. Kahn {\it The slice filtration and
mixed Tate motives}, Compositio Math. {\bf 142}  (2006), 907--936.
\bibitem{glr} B. Kahn {\it The Geisser-Levine method revisited and
algebraic cycles over a finite field}, Math. Ann. {\bf 324} (2002),
581--617.
\bibitem{FqX} B. Kahn {\it Zeta functions and motives}, Pure appl. Math.
Quarterly {\bf 5} (2009)  507-570 [2008].
\bibitem{modcyclnr} B. Kahn {\it Relatively unramified elements in cycle modules}, J. K-theory {\bf 7} (2011), 409--427.
\bibitem{kl} B. Kahn, M. Levine {\it Motives of Azumaya algebras}, J. Inst. Math. Jussieu {\bf 9} (2010), 481--599.
\bibitem{kmp} B. Kahn, J. P. Murre, C. Pedrini {\it On the transcendental part of the motive of a surface}, {\it in} Algebraic cycles and motives, Part II, LMS Series {\bf 344}, Cambridge University Press, 2007, 143--202.
\bibitem{birat} B. Kahn, R. Sujatha {\it Birational motives, I (preliminary
version)}, preprint, 2002, \url{http://www.math.uiuc.edu/K-theory/0596/}.
\bibitem{localisation} B. Kahn, R. Sujatha {\it A few localisation theorems},
Homology, Homotopy Appl. {\bf 9} (2007), 137--161.
\bibitem{Birat} B. Kahn, R. Sujatha {\it Birational geometry and localisation of
categories}, Doc. Math. -- Extra Volume Merkurjev (2015), 167--224.
\bibitem{birat-pure} B. Kahn, R. Sujatha {\it Birational motives, I: pure birational motives},
to appear in Annals of $K$-theory.
\bibitem{rqnr} B. Kahn, R. Sujatha {\it The derived functors of unramified cohomology}, preprint, \url{https://arxiv.org/abs/1511.07072}.
\bibitem{somekawa} B. Kahn, T. Yamazaki {\it  Voevodsky's motives and Weil reciprocity},  Duke Math. J. {\bf 162} (14) (2013), 2751--2796.
\bibitem{ka-sch} M. Kashiwara, P. Schapira Categories and sheaves, Grundl. der Math. Wiss. {\bf 332}, Springer, 2006.
\bibitem{keller} B. Keller {\it Deriving $DG$ categories}, Ann. Scient. \'Ec. Norm. Sup. {\bf 27} (1994), 63--102.
\bibitem{KK} H. Knaf, F.-V. Kuhlmann {\it Abhyankar places admit local uniformization in any characteristic}, Ann. Sci. \'Ec. Norm. Sup. {\bf 38} (2005), 833--846.
\bibitem{levineind} M. Levine {\it The indecomposable $K_3$ of fields}, Ann. Sci. \'Éc. Norm. Sup. {\bf 22} (1989), 255--344. 
\bibitem{mvw} C. Mazza, V. Voevodsky, C. Weibel Lecture notes on Motivic
  cohomology, Clay Math. Monographs {\bf 2}, AMS, 2006.
\bibitem{merk2} A.S. Merkurjev {\it Unramified elements in cycle modules}, J. London Math. Soc.  {\bf 78} (2008) 51--64.
\bibitem{msind} A. Merkurjev, A. Suslin {\it The group $K_3$ for a field (Russian)}, Izv. Akad. Nauk SSSR  {\bf 54} (1990), 522--545; translation in Math. USSR-Izv. {\bf 36} (1991), 541--565.
\bibitem{milne} J.S. Milne {\it Abelian varieties}, Ch. V of
Arithmetic Geometry (G. Cornell, J. Silverman, eds.), Springer,
1986, 103--150.
\bibitem{mv} F. Morel, V. Voevodsky {\it $\A^1$-homotopy theory of
schemes}, Publ. Math. IH\'ES {\bf 90} (1999), 45--143.
\bibitem{neeman} A. Neeman {\it The connection between the
$K$-theory localisation theorem of Thomason-Trobaugh and
Yao and the smashing subcategories of Bousfield and Ravenel}, Ann.
Sci. \'Ec. Norm. Sup. {\bf 25} (1992), 547--566. 
\bibitem{neeman2} A. Neeman Triangulated categories, Ann. Math. Studies
{\bf 148}, Princeton University Press, 2001.
\bibitem{Pauksztelllo} D. Pauksztelllo {\it Compact cochain objects in triangulated categories and co-$t$-structures}, Centr. Eur. J. of Math. {\bf 6} (2008), 25--42.
\bibitem{riou} J. Riou Th\'eorie homotopique des $S$-sch\'emas,
  m\'emoire de DEA, Paris 7, 2002, \url{http://www.math.u-psud.fr/~riou/dea/dea.pdf}. 
\bibitem{roberts} J. Roberts {\it Chow's moving lemma}, Appendix 2 to
{\it Motives} by Steven L. Kleiman, Algebraic geometry, Oslo 1970 (Proc.
Fifth Nordic Summer School in Math.), 89--96. Wolters-Noordhoff,
Groningen, 1972. 
\bibitem{rost} M. Rost {\it Chow groups with coefficients}, Doc. Math. {\bf 1} (1996), 319--393.
\bibitem{spaltenstein} N. Spaltenstein {\it Resolutions of unbounded
  complexes}, Compositio Math. {\bf 65} (1988), 121--154. 
\bibitem{spsz} M. Spie\ss, T. Szamuely {\it On the Albanese map for smooth quasi-projective varieties}, Math. Ann. {\bf 325} (2003), 1--17.
\bibitem{Somekawa} M. Somekawa {\it On Milnor $K$-groups attached to semi-abelian varieties}, $K$-theory {\bf 4} (1990), 105--119.
\bibitem{susvoe} A. Suslin, V. Voevodsky {\it Singular homology of
abstract algebraic varieties}, Invent. Math. {\bf 123} (1996), 61--94. 
\bibitem{verd} J.-L. Verdier Des cat\'egories d\'eriv\'ees des
cat\'egories ab\'eliennes, Ast\'erisque {\bf 239}, 1996.
\bibitem{voepre} V. Voevodsky {\it Cohomological theory of presheaves
  with transfers}, {\it in} E. Friedlander, A. Suslin, V. Voevodsky Cycles,
transfers and motivic cohomology theories, Ann. Math. Studies {\bf 143},
Princeton University Press, 2000, 188--187. 
\bibitem{voetri} V. Voevodsky {\it Triangulated categories of motives
over a field}, {\it in} E. Friedlander, A. Suslin, V. Voevodsky Cycles,
transfers and motivic cohomology theories, Ann. Math. Studies {\bf 143},
Princeton University Press, 2000, 188--238. 
\bibitem{voesli} V. Voevodsky {\it Open problems in the motivic stable
homotopy  theory, I}, preprint, 2000.
\bibitem{voecan} V. Voevodsky {\it Cancellation theorem}, Doc. Math. 2010, Extra volume: Andrei A. Suslin sixtieth birthday, 671--685.
\bibitem{zs} O. Zariski, P. Samuel Commutative Algebra, vol. II, van
Nostrand/Springer, 1960/1975.
\bibitem[SGA4-I]{SGA4} E. Artin, A. Grothendieck, J.-L. Verdier Th\'eorie
des topos et cohomologie \'etale des sch\'emas (SGA4), Vol. 1, Lect. Notes
in Math. {\bf 269}, Springer, 1972.
\end{thebibliography}
\end{document}